\documentclass[11pt]{amsart}
\usepackage[hmargin=2.5cm,vmargin=3.8cm]{geometry} 
 \usepackage[utf8]{inputenc}
 \usepackage[english]{babel}
\usepackage{tikz}
\usepackage{tikz-cd}
\usepackage{hyperref,amsmath,dsfont,amsthm,enumerate,xcolor,stmaryrd,mathrsfs,amssymb,enumitem,graphicx}
\usepackage{thmtools}
\usepackage{enumitem}
\usepackage{thm-restate}
\usepackage{subcaption}

\usepackage[all,knot,cmtip]{xy}
 \usepackage{color}
 \usepackage[backend=biber, style=alphabetic, sorting=nyt, url=false, doi=false, isbn=false, eprint=false, related=false ]{biblatex}
\renewcommand{\L}{\mathbb L}
\newcommand{\ot}{\otimes}
\newcommand{\Z}{\mathbb{Z}} 
\newcommand{\Q}{\mathbb{Q}}

\newcommand{\R}{\mathbb{R}}

\newcommand{\V}{\mathcal{V} }
\newcommand{\W}{\mathcal{W}}
\newcommand{\HH}{\mathcal{H}}
\newcommand{\X}{\mathcal{X}}

\newcommand{\LVW}{\L(V\oplus W)}

\DeclareMathOperator{\Aut}{Aut}
\DeclareMathOperator{\map}{map}
\DeclareMathOperator{\aut}{aut}
\DeclareMathOperator{\Diff}{Diff}
\DeclareMathOperator{\Homeo}{Homeo}
\DeclareMathOperator{\id}{id}
\DeclareMathOperator{\Der}{Der}
\DeclareMathOperator{\MC}{MC}
\DeclareMathOperator{\FI}{\mathsf{FI}}
\DeclareMathOperator{\Emb}{Emb}
\DeclareMathOperator{\weight}{weight}
\DeclareMathOperator{\stabdeg}{stab-deg}

\theoremstyle{definition}
\newtheorem{thm}{Theorem}[section]
\newtheorem{cor}[thm]{Corollary}
\newtheorem{prop}[thm]{Proposition}

\newtheorem{lemma}[thm]{Lemma}

\newtheorem{dfn}[thm]{Definition}

\newtheorem{rmk}[thm]{Remark}

\newtheorem*{thmmm}{Theorem}

\newcommand{\Addresses}{{% additional braces for segregating \footnotesize
  \bigskip
  \footnotesize

\textsc{Institut de mathématiques de Jussieu - Paris Rive Gauche, F-75013 Paris, France}\par\nopagebreak
  \textit{E-mail address:} \texttt{erikjlindell@gmail.com}\par\nopagebreak

  \vspace{1.5em}
  
\textsc{Department of Mathematics, Stockholm University, SE-106 91 Stockholm, Sweden}\par\nopagebreak
 
  \textit{E-mail address:} \texttt{basharsaleh1@gmail.com}

}}

\begin{document}
\title{Representation Stability for Homotopy Automorphisms} 
\author{Erik Lindell and Bashar Saleh} 
\date{}
\maketitle

\begin{abstract}
    We consider in parallel pointed homotopy automorphisms of iterated wedge sums of finite CW-complexes and boundary relative homotopy automorphisms of iterated connected sums of manifolds minus a disk. Under certain conditions on the spaces and manifolds, we prove that the rational homotopy groups of these homotopy automorphisms form finitely generated $\FI$-modules, and thus satisfy representation stability for symmetric groups, in the sense of Church and Farb. We also calculate explicit bounds on the weights and stability degrees of these $\FI$-modules.
\end{abstract}

\setcounter{tocdepth}{1}

\parskip=2pt  

\tableofcontents

\parskip=10pt

\section{Introduction} Pointed homotopy automorphisms of iterated wedge sums of spaces and boundary relative homotopy automorphisms of iterated connected sums of manifolds minus a disk, come with  stabilization maps that  yield questions of whether the homology groups or the homotopy groups of these homotopy automorphisms stabilize in any sense. Previously, Berglund and Madsen \cite{BM14} have proven rational homological stability for homotopy automorphisms of iterated connected sums of higher dimensional tori $S^n\times S^n$, for $n\ge 3$,  and these results were later expanded by Grey \cite{grey} and Stoll \cite{stoll}, for homotopy automorphisms of iterated connected sums of certain manifolds of the form $S^n\times S^m$, for $n,m\ge 3$. 

In this paper we instead study the rational \textit{homotopy groups} of the homotopy automorphisms in question, which we consider as based spaces with the identity map as the base point. These homotopy groups do not stabilize in the traditional sense. Instead, we show that they satisfy a different kind of stability, known as \textit{representation stability}. In the two cases we study in this paper, we consider sequences of rational homotopy groups, which in step $n$ are representations of the symmetric group $\Sigma_n$. For such representations, there is a consistent way to name the irreducible representations for arbitrary $n$, and representations stability essentially means that as $n$ tends to infinity, the decomposition into irreducible representations eventually becomes constant. 

Representation stability was introduced by Church and Farb \cite{CF13} and later further developed by Church, Ellenberg and Farb \cite{CEF15}, who showed that for representations of symmetric groups, this notion can be encoded by so called $\FI$-modules, which are functors from the category of finite sets and injections to the category of vector spaces. The stable range of representation stability corresponds to \textit{stability degree} and \textit{weight} of the corresponding $\FI$-module.

We review $\FI$-modules and representation stability in more detail in Section \ref{sectionFImodules}. The first main result of this paper is the following:

\begin{restatable}{thmx}{theoremone}\label{theorem1}
 Let $(X,*)$ be a pointed, simply connected space with the homotopy type of a finite CW-complex and let $X_S:=\bigvee^S X$, for any finite set $S$. For each $k\ge 1$, the functor
$$S\mapsto\pi_{k}^{\mathbb{Q}}(\aut_*(X_S)) $$
is an $\FI$-module. If $H_n(X,\Q)=0$ for $n\ge d$, this $\FI$-module is of weight $\le k + d-1$ and stability degree $\le k+d$.\hfill $\diamond$
\end{restatable}

For the analogous theorem for connected sums, we need the notion of a boundary relative homotopy automorphism of a manifold $N$ (with boundary). A boundary relative homotopy automorphism of $N$ is a homotopy automorphism of $N$ that preserves the boundary $\partial:=\partial N$ pointwise. The boundary relative homotopy automorphisms of $N$ form a topological monoid, with respect to composition, which we will denote by $\aut_\partial(N)$. 

Let $M=M^d$ be a closed, oriented, $d$-dimensional manifold. For any finite set $S$, we let $M_S$ denote the $S$-fold connected sum of $M$ with itself, with an open $d$-disk removed; $M_S = \#^S M\smallsetminus \mathring D^d$. For $\mathbf n = \{1,2,\dots,n\}$, we denote $M_{\mathbf n}$ simply by $M_n$. A homotopy automorphism of $M_n$ does not extend to
a homotopy automorphism of $M_{n+1}$ in any canonical way in general. However, boundary relative
homotopy automorphisms of $M_n$ extend by identity to a boundary relative homotopy automorphism of $M_{n+1}$.
In particular, there is a stabilization map
$$
s_n\colon \aut_\partial(M_n)\to \aut_\partial(M_{n+1}).
$$

By picking some base point in the boundary of $M_1$, there is a deformation retract $M_S\xrightarrow\simeq \bigvee^S M_1$ (see e.g. \cite[§ 3.1.2]{felixamg}), where the wedge sum is taken along this base point. It follows by Theorem \ref{theorem1}  that there is an FI-module given on objects by $S\mapsto\pi_k(\aut_*(M_S)) \cong \pi_k(\aut_*( \bigvee^S M_1))$. For any finite set $S$ we have an obvious inclusion map  $\aut_\partial(M_S)\hookrightarrow\aut_*(M_S)$, so we may ask whether we can find an $\FI$-module given by $S\mapsto\pi_k(\aut_\partial(M_S))$ that make these maps into a morphism of $\FI$-modules, i.e.\ a natural transformation of functors. We will refer to this a ``lifting'' the $\FI$-module structure. In our second main theorem, we address this problem:

\begin{restatable}{thmx}{theoremtwo}\label{theorem2}
Let $M = M^d$ be a closed, simply connected, oriented $d$-dimensional manifold. With $M_S$ defined as above, we have the following:
\begin{itemize}
    \item[(a)]  For each $k\ge 1$, the $\FI$-module
$$
S\mapsto \pi_k\left(\aut_*\left({\bigvee}^S M_1\right)\right) \cong  \pi_k(\aut_*(M_S)) 
$$
lifts to an $\FI$-module 
$$
S\mapsto  \pi_k(\aut_\partial(M_S)) 
$$
sending the standard inclusion $\mathbf n\to \mathbf{n+1}$ to the map $\pi_{k}(\aut_\partial(M_n))\to \pi_{k}(\aut_\partial(M_{n+1}))$
induced by the stabilization map $s_n$.
\item[(b)] The rationalization of this $\FI$-module is of weight $\le k+d-2$ and stability degree $\le k+d-1$.\hfill$\diamond$ 
\end{itemize} 

\end{restatable}

\iffalse

\begin{rmk}
An interesting question is whether Theorem \ref{theorem2} can be extended to hold for manifolds with boundary, with a disk removed, in which only the boundary of the removed disk is required to be fixed by the homotopy automorphisms. In this case, Theorem \ref{theorem1} would follow as a special case, by embedding $X$ in an Euclidean space of sufficiently high dimension, taking $M$ to be the closure of some regular neighborhood of this embedding and then removing a disk around the base point of $X$. 
\end{rmk}
\fi

\begin{rmk}
Theorem \ref{theorem1} and Theorem \ref{theorem2} are somewhat analogous to those for unordered configuration spaces of manifolds. Rational homological stability for unordered configuration spaces of arbitrary connected manifolds was proven by Church \cite{Church12}, following integral results for open\footnote{Integral homological stability is known not to hold for closed manifolds. A simple counterexample is given already by the $2$-sphere $S^2$, where $H_1(B_n(S^2),\Z)\cong \Z/(2n-2)\Z$ (see for example \cite[Theorem 1.11]{Birman75}).} manifolds by Arnol'd \cite{Arnold69}, MacDuff \cite{McDuff75} and Segal \cite{Segal79}. It was later proven by Kupers and Miller \cite{KupersMiller18} that the rational \textit{homotopy groups} of unordered configuration spaces on connected, simply connected manifolds of dimension at least 3, satisfy representation stability. 

Homotopy automorphisms of iterated wedge sums of \textit{spheres} has been studied by Miller, Patzt and Petersen \cite{miller2019representation}. Using representation stability, they prove that for $d\ge 2$, the sequence $\{B\aut\left(\bigvee_{i=1}^n S^d\right)\}_{n\ge 1}$ satisfies homological stability with $\Z\left[\frac{1}{2}\right]$-coefficients, which proves homological stability with the same coefficients for $\{B\mathrm{GL}_n(\mathbb{S})\}_{n\ge 1}$, where $\mathbb{S}$ is the sphere spectrum. These representation stability results are neither weaker nor stronger than Theorem \ref{theorem1}, since on one hand they work with $\Z\left[\frac{1}{2}\right]$-coefficients and on the other hand we work with wedge sums of more general CW-complexes than spheres.

For a simply connected  $d$-dimensional manifold $M$, with boundary $\partial M\cong S^{d-1}$, the rational homotopy theory of $\aut_\partial(M)$ has been thoroughly studied by Berglund and Madsen \cite{BM14}, whose results we will use.\hfill$\diamond$
\end{rmk}

As a by-product of the techniques used for proving Theorem \ref{theorem2} (a)  we get the following:

\begin{thmmm}
Let $M$ be a closed, oriented, simply connected $d$-dimensional manifold such that the reduced homology of $M\smallsetminus\mathring D^d$ is non-trivial. Given a subspace $A\subseteq \partial M_n$, possibly empty, such that $A\subset M_n$ is a cofibration, then  the groups $\pi_0(\aut_A(M_n))$, $\pi_0(\Diff_A(M_n))$ and  $\pi_0(\Homeo_A(M_n))$ contain a subgroup isomorphic to $\Sigma_n$. 
\hfill$\diamond$
\end{thmmm}

\subsection{Structure of the paper.} In Section 2, we review the necessary background on $\FI$-modules that we will need. The reader familiar with $\FI$-modules may skip directly  to Section \ref{sec:FI-models}, where we introduce the notion of $\FI$-\textit{Lie models} of pointed $\FI$-spaces, which is of key importance for proving the main theorems. In Section 3 we review necessary rational homotopy theory for homotopy automorphisms needed for proving the main theorems.  In Section 4, we study homotopy automorphisms of wedge sums and prove Theorem \ref{theorem1}. In Section 5, we study homotopy automorphisms of connected sums and prove Theorem \ref{theorem2}.

\subsection{Acknowledgements} The authors would like to thank Dan Petersen and Alexander Berglund for many useful comments and suggestions. We would also like to thank the referees, whose careful reading and useful comments have improved the paper greatly from its first version.

Subsection \ref{subsec:integral-thm-b}, which treats integral homotopy theory for relative homotopy automorphisms, has developed greatly since the first preprint version of this paper, thanks to many other people. Our decision to consider the integral homotopy groups of the homotopy automorphisms of iterated connected sums is inspired by an answer by Ryan Budney to a question by the second named author at MathOverflow. The method used to prove Theorem \ref{theorem2} (a) was suggested by Manuel Krannich, who has also come with several other very helpful comments. In addition, he was in the committee for the PhD defense of the first author, where he, together with Fabian Hebestreit, pointed out several minor errors in the paper and came with some very helpful suggestions for improvements.

The second author was supported by the Knut and Alice Wallenberg Foundation through grant no. 2019.0521.

\section{Representation stability, FI-modules and Lie models of FI-spaces}\label{sec:fi} \label{sectionFImodules}

\subsection{Conventions.} Throughout the paper, we will use $R$ to denote a commutative ring, which we will assume to be Noetherian for convenience. We will mainly work over the field $\Q$, so unless otherwise specifically stated, all vector spaces are over $\Q$. We will use ``dg'' to abbreviate the term \textit{differential graded}. Throughout the paper, $\FI$ denotes the category of finite sets with injective maps as morphisms. 

If $S$ is a finite set, we will use $|S|$ to denote its cardinality, and we will write $\Sigma(S):=\Aut_{\FI}(S)$ for the symmetric group on $S$. If $S=\mathbf{n}:=\{1,2,\ldots,n\}$ we will simply write $\Sigma(S)=\Sigma_n$ for brevity.

Recall that the irreducible $\Q$-representations of $\Sigma_n$ are indexed by partitions of weight $n$, i.e.\ sequences of non-negative integers $\lambda=(\lambda_1\ge \lambda_2\ge \cdots\lambda_l\ge 0\ge \cdots)$ such that $|\lambda|=\lambda_1+\lambda_2+\cdots=n$. We will denote the corresponding $\Q$-representation by $V_\lambda$. For any $k\ge n+\lambda_1$, we also define the \textit{padded} partition $\lambda[k]:=(k-n\ge \lambda_1\ge\lambda_2\ge\cdots)$ and write $V(\lambda)_k:=V_{\lambda[k]}$.

\subsection{Representation stability} Before we introduce the language of $\FI$-modules, let us recall the original notion of representation stability, which is formulated in terms of consistent sequences of $\Sigma_n$-representations.

\begin{dfn}\label{dfn:consistentSequence}
Let $R$ be a commutative ring. A \textit{consistent sequence} of $\Sigma_n$-representations over $R$ is a sequence $\{V^n,\phi^n\}$ where $V^n$ is a $R[\Sigma_n]$-module and $\phi^n:V^n\to V^{n+1}$ is an $\Sigma_n$-equivariant map (where $V^{n+1}$ is considered a $R[\Sigma_n]$-module through the standard inclusion $\Sigma_n\hookrightarrow\Sigma_{n+1}$).
\end{dfn}

If $R=\Q$, we may define (uniform) \textit{representation stability} for such a sequence as follows:

\begin{dfn}
A consistent sequence of rational $\Sigma_n$-representations $\{V^n,\phi^n\}$ is said to be uniformly representation stable with stable range $n\ge N$ if it satisfies the following for all $n\ge N$:\begin{enumerate}
    \item the map $\phi^n$ is injective,
    \item the image of $\phi^n$ generates $V^{n+1}$ as a $\Sigma_{n+1}$-representation,
    \item for each partition $\lambda$ the multiplicity of $V(\lambda)_n$ in $V^n$ is independent of $n$.
\end{enumerate}
\end{dfn}

Next, we will introduce $\FI$-modules, and recall how representation stability is encoded in that language.

\subsection{$\FI$-modules}\label{Section-FI-modules} We first introduce the notion of an $\FI$-object in an arbitrary category.

\begin{dfn}
Let $\mathcal{C}$ be a category. A functor $\FI\to\mathcal{C}$ is called an $\FI$-object in $\mathcal{C}$.\hfill $\diamond$
\end{dfn}

Let us review the kinds of $\FI$-objects that will be of interest to us:\begin{itemize}
    \item An $\FI$-object in $\mathrm{(gr)}\mathrm{Mod}_R$, the category of ($\mathbb{Z}$-graded) $R$-modules, will be called a (graded) $\FI$-$R$-\textit{module}. An $\FI$-object in $\mathrm{dgMod}_R$, the category of differential graded $R$-modules, will be called a \textit{dg }$\FI$-$R$-module. For a dg $\FI$-$R$-module $\V $, we will write $H_*(\V)$ for the composition with the homology functor and refer to it as the homology of $\V $.
    \item An $\FI$-object in $\mathrm{dgLie}_R$, the category of dg Lie algebras, over $R$, will be called a dg $\FI$-$R$-Lie algebra.
    \item An $\FI$-object in $\mathrm{Top}_*$, the category of pointed topological spaces, will be called a \textit{pointed} $\FI$-\textit{space}. If $\mathcal{P}$ is a  property of pointed topological spaces, such as being simply connected, we will say that a pointed $\FI$-space $\mathcal{X}$ has property $\mathcal{P}$ if $\mathcal{X}(S)$ has property $\mathcal{P}$, for every finite set $S$. If $\mathcal{X}$ is a pointed $\FI$-space with $\pi_1(\X(S))$ being  abelian for every finite set $S$, composing with the (rational) homotopy groups functor $\pi_{*}$ (resp. $\pi_{*}^\Q$) naturally gives us a graded $\FI$-$\Z$-module (resp.  graded $\FI$-$\Q$-module). We will simply write $\pi_{*}(\mathcal{X})$ (resp. $\pi_{*}^\Q(\X)$) for this composite functor and refer to it as the (rational) homotopy groups of $\mathcal{X}$.
\end{itemize}

We will generally consider the first two examples for $R=\Q$ and $R=\Z$. If the ring is clear from context, or if the choice of $R$ is not important, we will generally drop it from the notation. 

Now let us recall some basics from the theory of $\FI$-modules. Since the category of (graded) $R$-modules is abelian, the category of (graded) $\FI$-$R$-modules inherits this structure, which means that there are natural notions of (graded) $\FI$-$R$-submodules as well as quotients, direct sums and tensor products of (graded) $\FI$-$R$-modules, all defined pointwise (cf. \cite[Remark 2.1.2]{CEF15}). 

\begin{rmk}
Note that any $\FI$-$R$-module $\mathcal V$ gives rise to a consistent sequence $\{V^n:=\mathcal V(\mathbf n), \phi^n:=\V (\mathbf{n}\hookrightarrow \mathbf{n+1})\}$ of $R[\Sigma_n]$-modules, where $\mathbf{n}\hookrightarrow \mathbf{n+1}$ is the standard inclusion.\hfill$\diamond$
\end{rmk}

\begin{rmk}
 Not every consistent sequence  arises from an FI-module (cf. \cite[Remark 3.3.1]{CEF15}). 
\strut\hfill$\diamond$
\end{rmk}

Sometimes it will be more convenient to work with consistent sequences than with $\FI$-modules. For this purpose, the following lemma is important:

\begin{lemma}[\text{cf. \cite[Remark 3.3.1]{CEF15}}]\label{lemma:consistentSeqFromFI-module}
A consistent sequence $\{V^n,\phi^n\}$ is induced by some $\FI$-module if and only if for every $\sigma\in \Sigma_{n+k}$ with $\sigma|_{\mathbf n}=\id$ acts trivially on 
$$\mathrm{im}\left(\phi^{n+k-1}\circ\cdots\phi^n\colon V^n\to V^{n+k}\right).$$
If two $\FI$-modules give rise to isomorphic consistent sequences, then the two $\FI$-modules are isomorphic.\hfill $\diamond$
\end{lemma}

The main property of $\FI$-modules that will be of interest to us is \textit{finite generation}, since this is what encodes representation stability:

\begin{dfn}
Let $\mathbf{n}:=\{1,2,\ldots, n\}$. A (graded) $\FI$-$R$-module $\V$ is said to be \textit{finitely generated} if there exists a finite set $S\subset\bigsqcup_{n\ge 1} \V (\mathbf{n})$ such that there is no proper (graded) $\FI$-$R$-submodule $\mathcal{W}$ of $\V$ such that $S\subset\bigsqcup_{n\ge 1}\mathcal{W}(\mathbf{n})$.\hfill $\diamond$
\end{dfn}

Now we can describe how representation stability relates to $\FI$-modules:

\begin{thm}[\text{\cite[Theorem 1.13]{CEF15}}]
An $\FI$-$\Q$-module $\V$ is finitely generated if and only if the consistent sequence $\{V^n:=\V (\mathbf{n})\}$ is uniformly representation stable and each $V^n$ is finite dimensional. \hfill$\diamond$
\end{thm}

What makes working with the category of $\FI$-$R$-modules for any Noetherian ring $R$ particularly useful is that it is \textit{Noetherian}, i.e.\ an $\FI$-$R$-submodule of such a finitely generated $\FI$-$R$-module is itself finitely generated (see \cite[Theorem 1.3]{CEF15} and \cite[Theorem A]{CEFN14}). Finite generation is also preserved by tensor products and quotients. This means that to prove that an $\FI$-$R$-module is finitely generated, it suffices to show that it is a subquotient of a tensor product of some $\FI$-$R$-modules that are more obviously finitely generated.

Since we want to use rational homotopy theory to prove our results, we need to consider graded $\FI$-modules in our proofs. For this reason we will need the following definition:

\begin{dfn}
 If $\V$ is a graded $\FI$-$R$-module and $m\in\mathbb{Z}$, let $\V_m$ be the degree $m$ part of $\V$, i.e.\ the post-composition with the functor $\mathrm{grVect}_\Q\to\mathrm{Vect}_\Q$ given by sending a graded vector space to its degree $m$ part. If $\V_m=0$ for $m\le m'$ (resp. $m\ge m')$ we say that $\V $ is concentrated in degrees above (resp. below) $m'$. Such a graded $\FI$-module is called bounded from below (above).\hfill$\diamond$
\end{dfn}

\iffalse
\begin{dfn}
A graded $\FI$-module $\V$ is said to be of \textit{finite type} if $\V_m$ is finitely generated for each $m\ge 0$.\hfill$\diamond$
\end{dfn}

By the Noetherian property of $\FI$-modules, a graded $\FI$-submodule of a graded $\FI$-module of finite type is also of finite type. It is also clear that quotients of graded $\FI$-modules of finite type are of finite type. Note that this implies that if a dg $\FI$-module is of finite type, considered just as a graded $\FI$-module, then so is its homology. For tensor products to preserve finite type, we however need to add extra assumptions on the tensor factors:

\begin{lemma}
Let $\mathcal V$ and $\mathcal W$ be graded $\FI$-modules of finite type where both are either bounded from below or from above. Then $\mathcal V\ot\mathcal W$ is also a graded $\FI$-module of finite type.\hfill$\diamond$ 
\end{lemma}
\fi

\subsection{Weight and stability degree} For the rest of Section \ref{sec:fi}, we will assume that $R=\Q$. We have seen how finite generation of $\FI$-$\Q$-modules corresponds to representation stability of the corresponding consistent sequence of rational $\Sigma_n$-representations, but in order to make quantitative statements about stability ranges we need to introduce the \textit{weight} and \textit{stability degree} of such $\FI$-modules.

Recall that if $V$ is a $\Sigma_n$-representation, $(V)_{\Sigma_n}$ denotes the quotient of \textit{coinvariants} of $V$. For an $\FI$-module $\V $, this allows us to define a sequence $\{\phi_a(\V)^n\}$ of vector spaces and maps between them, for each $a\ge 0$, as follows: let $\phi_a(\V )^n:=\left(\V (\mathbf{a}\sqcup\mathbf{n})\right)_{\Sigma_n}$. Any inclusion $\iota:\mathbf{n}\hookrightarrow\mathbf{n+1}$ gives us an inclusion $\mathrm{id}\sqcup\iota:\mathbf{a}\sqcup\mathbf{n}\hookrightarrow\mathbf{a}\sqcup(\mathbf{n+1})$, inducing a map $\phi_a(\V )^n\to\phi_a(\V )^{n+1}$. Since we quotient by $\Sigma_{n+1}$, the choice of inclusion $\iota$ does not matter.

With this, we can define the \textit{stability degree} of an $\FI$-module as follows:

\begin{dfn}[\text{ \cite[Definition 3.1.3]{CEF15}}]
The \textit{injectivity degree} $\text{inj-deg}(\V)$ (\textit{surjectivity degree} $\text{surj-deg}(\mathcal{V})$) of an $\FI$-module $\V $ is the smallest $s\ge 0$ such that for all $a\ge 0$, the map $\phi_a(\V )^n\to\phi_a(\V )^{n+1}$, defined as above, is injective (surjective) for all $n\ge s$ (and if no such $s$ exists we set the degree to $\infty$). We define the \textit{stability degree} $\stabdeg(V)$  of $\mathcal{V}$ to be the maximum of the injectivity and surjectivity degrees.\hfill $\diamond$
\end{dfn}

Next, we define the \textit{weight} of an $\FI$-module.

\begin{dfn}
The weight of an $\FI$-module $\V $, which we denote by $\weight(\V )$, is the maximum weight $|\lambda|$ over all $V(\lambda)_n$ appearing in the $\Sigma_n$-representation $\V (\mathbf{n})$, if such a maximum exists. If no maximum exists we set $\weight(\V )=\infty$ and if the $\FI$-module is zero we set it to zero.\hfill $\diamond$
\end{dfn}

These definitions are relevant because of their relation to representation stability, which may now be stated as follows:

\begin{prop}[\text{\cite[Proposition 3.3.3]{CEF15}}]
Let $\V $ be an $\FI$-module. The consistent sequence $\{V^n,\phi^n\}$ determined by $\V $ is uniformly representation stable with stable range $n\ge\weight(\V )+\stabdeg(\V )$.\hfill $\diamond$
\end{prop}

\begin{rmk}
Note that in particular, this proposition implies that if an $\FI$-module has finite weight and stability degree, it is finitely generated. For this reason we will only be working with weight and stability degree going forward. It should be noted, however, that due to the Noetherian property of $\FI$-modules, it is possible to freely take submodules and quotients and preserve finite generation. This is not the case for stability degree, as we will see, making it much easier to prove finite generation than to obtain an explicit bound on stability degree. 
\end{rmk}

Let us recall some useful properties of weight and stability degree. First, the following is immediate from the definitions:

\begin{prop}
Let $\V^1$ and $\V^2$ be $\FI$-modules. Then $\weight(\V^1\oplus\V^2)\le\max(\weight(\V^1),\weight(\V^2))$ and $\stabdeg(\V^1\oplus\V^2)\le\max(\stabdeg(\V^1),\stabdeg(\V^2))$.
\end{prop}

Next, we will recall how weight and stability degree behaves under taking tensor products.

\begin{prop}\label{prop:weightstab-tensor}
Suppose that $\V^1,\V^2,\ldots,\V^k$ are $\FI$-modules with stability degrees $ \le r_1, r_2,\ldots, r_k$ and weights $\le s_1, s_2,\ldots,s_k$, respectively. Then 
$$\weight(\V^1\otimes\cdots\otimes \V^k)\le s_1+\cdots+s_k$$
and 
$$\stabdeg(\V^1\otimes\cdots\otimes\V^k)\le \max(r_1+s_1,\ldots,r_k+s_k,s_1+\cdots+s_k).$$
\end{prop}

\begin{proof}
The first part is \cite[Proposition 3.2.2]{CEF15}, while the second part is \cite[Proposition 2.23]{KupersMiller18}.
\end{proof}

We also need to know how stability degree and weight behaves when taking submodules and quotients.

\begin{prop}\label{prop:weightstab-subquotient}
Let $\V$ be an $\FI$-module, $\mathcal{W}$ an $\FI$-submodule of $\V$. Then $\weight(\mathcal{W})\le\weight(\V)$ and $\weight(\mathcal{V}/\mathcal{W})\le \weight(\V)$. If in addition $\V$ is such that $\V(S)$ is finite dimensional for every finite set $S$, we have the following: \begin{enumerate}
    \item $\text{inj-deg}(\mathcal{W})\le\text{inj-deg}(\mathcal{V})$,
    \item $\text{surj-deg}(\mathcal{V}/\mathcal{W})\le\text{surj-deg}\mathcal({V})$,
    \item if $\text{inj-deg}(\V)\le r$ and $\text{surj-deg}(\W)\le r$, then $\text{inj-deg}(\V/\mathcal{W})\le r$,
    \item if $\text{surj-deg}(\V)\le r$ and $\text{inj-deg}(\V/\W)\le r$, then $\text{surj-deg}(\mathcal{W})\le r$,
\end{enumerate} 
\end{prop}

\begin{proof}
The first part follows directly by the definition of weight.  (1) and (2) are \cite[Lemma 3.1.6]{CEF15}. 

To prove (3) and (4), note that for each $a\ge 0$, $\phi_a$ defines a functor from the category of $\FI$-modules to the category of sequences of vector spaces and linear maps and this functor is exact. Thus (3) and (4) follow from the following respective propositions from linear algebra: if $f:V\to W$ is a linear map of finite dimensional vector spaces, $V'\subseteq V$ and $W'\subseteq W$ are subspaces, $f':V'\to W'$ is a linear map such that $f'(v)=f(v)$ for all $v\in V'$ and $f/f':V/V'\to W/W'$ is the induced map between the quotients, then
\begin{enumerate}[start=3,label={ (\arabic*'):}]
    \item if $f$ is injective and $f'$ is surjective, then $f/f'$ is injective,
    \item if $f$ is surjective and $f/f'$ is injective, then $f'$ is surjective.
\end{enumerate}
Proving these are both simple exercises in linear algebra and therefore left to the reader.
\end{proof}

Note however that given \textit{only} the stability degree of an $\FI$-module we can in general not say anything about the stability degree of its $\FI$-submodules or quotients. However, if an $\FI$-module $\V$ is isomorphic to \textit{both} an $\FI$-submodule of an $\FI$-module and a quotient of an $\FI$-module, for both of which we have bounds on the stability degree, we can use the proposition above to determine a bound on $\stabdeg(\V)$. This will be the case for an important class of $\FI$-modules that we consider in Subsection \ref{subsec:schur}. Note that in particular, we get the following corollary:

\begin{cor}\label{cor:stabdeg-summand}
Suppose $\mathcal{W}$ is an $\FI$-module which is a direct summand of another $\FI$-module $\V$, i.e.\ that there exists a third $\FI$-module $\mathcal{U}$ such that $\V\cong\mathcal{W}\oplus \mathcal{U}$. Then $\stabdeg(\mathcal{W})\le\stabdeg(\V)$.
\end{cor}

Finally, we need a way to determine the weight and stability degree in each degree when taking the homology of a differential graded $\FI$-module. We will prove the following more general statement (cf. \cite[Proposition 2.19]{KupersMiller18}):

\begin{prop}\label{prop:weightstab-homology}
Let $\mathcal{U}\overset{f}{\to}\V\overset{g}{\to}\mathcal{W}$ be a sequence of $\FI$-modules and morphisms of $\FI$-modules, such that $\mathcal{U}(S),\mathcal{W}(S)$ and $\V(S)$ are finite-dimensional for all $S\in\FI$, and $g\circ f=0$. Then $\weight(\mathrm{ker}(g)/\mathrm{im}(f))\le \weight(\mathcal{W})$ and if all three $\FI$-modules have stability degree $\le r$, then $\stabdeg(\mathrm{ker}(g)/\mathrm{im}(f))\le r$. In particular, if $\V$ is a dg $\FI$-module such that $\V_{m}(S)$ is finite dimensional for each $m$ and finite set $S$, then $\weight(H_m(\V))\le \weight(\V_m)$ and if $\stabdeg(\V_i)\le r$ for $i\in \{m-1,m,m+1\}$, we have $\stabdeg(H_m(\V))\le r$.
\end{prop}

\begin{proof}
The first part follows directly from the first part of Proposition \ref{prop:weightstab-subquotient}. We prove the second part by showing that the homology has injectivity and surjectivity degree $\le r$. For injectivity degree, note that $\mathrm{ker}(g)$ has injectivity degree $\le r$, by Proposition \ref{prop:weightstab-subquotient}(1), since it is a $\FI$-submodule of $\mathcal{V}$. Furthermore, since the category of $\FI$-modules is abelian, we have $\mathrm{im}(f)\cong \mathcal{U}/\mathrm{ker}(f)$, which has surjectivity degree $\le r$, by Proposition \ref{prop:weightstab-subquotient}(2). Thus it follows from Proposition \ref{prop:weightstab-subquotient}(3) that $\mathrm{ker}(g)/\mathrm{im}(f)$ has injectivity degree $\le r$. 

For surjectivity degree, we argue similarly as follows: the injectivity degree of $\mathrm{im}(g)\le r$ by Proposition \ref{prop:weightstab-subquotient}(1), and since $\mathrm{im}(g)\cong\mathcal{V}/\mathrm{ker}(g)$, we thus get by Proposition \ref{prop:weightstab-subquotient}(4) that $\mathrm{ker}(g)$ has surjectivity degree $\le r$. Thus the quotient $\mathrm{ker}(g)/\mathrm{im}(f)$ does as well, by Proposition \ref{prop:weightstab-subquotient}(2).
\end{proof}

\subsection{$\FI^\#$-modules} Many $\FI$-modules appearing ``naturally'' actually have additional structure, which may be encoded using the notion of an $\FI^\#$-module. The category $\FI^\#$ has the same objects as $\FI$, but the morphisms $S\to T$ are given by a pair of subsets $A\subset S$ and $B\subset T$ and a bijection $A\to B$. We call these \textit{partial injections}. An $\FI^\#$-object in a category $\mathcal{C}$ is simply a functor $\FI^\#\to\mathcal{C}$. Since $\FI$ is a subcategory of $\FI^\#$, any $\FI^\#$-object has an underlying $\FI$-object, so all the notions defined in the previous sections can be defined for (graded) $\FI^\#$-modules, by simply considering the underlying (graded) $\FI$-module.

The reason we will consider $\FI^\#$-modules is because there is a natural way to define \textit{duals} in this category. Note that the category $\FI^\#$ is naturally isomorphic to its opposite category, simply by taking the inverse of the bijection (see the end of \cite[Remark 4.1.3]{CEF15}). This allows us to make the following definition:

\begin{dfn}
If $\V:\FI^\#\to\mathrm{Vect}_\Q$, we define the \textit{dual} $\FI^\#$-module $\V^*$ as the composite functor
\[\begin{tikzcd}
   \FI^\#\arrow[r,"\cong"]&(\FI^\#)^{\text{op}}\arrow[r,,"\V^{\text{op}}"]&\mathrm{Vect}_\Q^{\text{op}}\arrow[rr,"{\mathrm{Hom}_\Q(-,\Q)}"]&&\mathrm{Vect}_\Q.
\end{tikzcd}\]
\end{dfn}

\subsection{Schur functors}\label{subsec:schur} The graded $\FI$-modules that we will study will be constructed by composing \textit{Schur functors} with simpler graded $\FI$-modules, which is why they are of finite type. In this section we will review what we mean by Schur functors in this context and their properties when composed with graded $\FI$-modules.

If $\lambda$ is a partition of $k\ge 0$, we define the Schur functor $\mathbb{S}_\lambda:\mathrm{gr}\mathrm{Vect}_\mathbb{Q}\to \mathrm{gr}\mathrm{Vect}_\mathbb{Q}$ on objects by
$$V\mapsto S^\lambda\otimes_{\Sigma_k} V^{\otimes k},$$
considering $V^{\otimes k}$ with the standard $\Sigma_k$-action and considering $S^\lambda$ as a graded vector space concentrated in degree 0. Another definition, which gives an isomorphic functor, is that $\mathbb{S}_\lambda(V)$ is given by the composition of the $k$th tensor power functor with the action of a certain idempotent operator $c_\lambda\in\Q[\Sigma_k]$, known as a \textit{Young symmetrizer}, acting on $V^{\otimes k}$ (see \cite{fultonharris} for a definition). This characterizes $\mathbb{S}_\lambda(V)$ as a subrepresentation of $V^{\otimes k}$.

If $W$ is a finite dimensional graded $\Sigma_k$-representation, we more generally define its associated Schur functor by
$$V\mapsto W\otimes_{\Sigma_k}V^{\otimes k},$$ 
and denote it by $\mathbb{S}_W$. Note that since $W$ is finite dimensional, this functor decomposes as a direct sum of Schur functors $\mathbb{S}_\lambda$ (possibly shifted in degree).

Even more generally, given a symmetric sequence $W=(W(1),W(2),\ldots)$ of (graded) vector spaces, i.e.\ a sequence in which $W(k)$ is a graded $\Sigma_k$-representation, we can associate to it the endofunctor $\bigoplus_{k\ge 1} \mathbb{S}_{W(k)}\circ\V$ of $\mathrm{gr}\mathrm{Vect}_\Q$,
which we will denote by $\mathbb{S}_W$ and call the Schur functor associated to $W$.

Schur functors are of interest to us, since they preserve stability degree and weight in the following way:

\begin{prop}\label{propSchurFT}
Let $W=(W(1),W(2),\ldots)$ be a symmetric sequence of graded vector spaces, where each $W(k)$ is finite dimensional and concentrated in non-negative degree, and let $\V:\FI\to\mathrm{grVect}_\Q$ be a graded $\FI$-module, such that $\V(S)$ is concentrated in strictly positive degrees for every $S\in\FI$. Suppose that $\V(S)$ is finite dimensional in each degree and that $\weight(\V_i)\le s$ and $\stabdeg(\V_i)\le r$ for all $i\le m$. Then we have $\weight\left((\mathbb{S}_W\circ \V)_m\right)\le m s$ and $\stabdeg\left((\mathbb{S}_W\circ\V)_m\right)\le\max(r+s,m s)$.
\end{prop}

\begin{proof}
By definition $\mathbb{S}_W\circ \V$ decomposes as the direct sum
$$\bigoplus_{k\ge 1}\mathbb{S}_{W(k)}\circ \V$$
and we may decompose each summand further as
$$\mathbb{S}_{W(k)}\circ\V= \bigoplus_{j\ge 0}\bigoplus_{i\ge 1} W(k)_j\otimes_{\Sigma_k} (\V^{\otimes k})_i.$$
Since $W(k)$ is concentrated in non-negative degree and $\V$ is concentrated in positive degree, it follows that $\mathbb{S}_{W(k)}\circ\V$ is concentrated in degrees $\ge k$. We thus have
$$\left(\mathbb{S}_W\circ\V\right)_m= \bigoplus_{k=1}^m \bigoplus_{i+j=m} W(k)_j\otimes_{\Sigma_k}(\V^{\otimes k})_i$$
By Corollary \ref{cor:stabdeg-summand}, it thus suffices to find bounds on the weight and stability degree of $W(k)_j\otimes_{\Sigma_k}(V^{\otimes k})_i$, for all $k\le m$ and all $i,j$ such that $i+j=m$. By definition, this is a quotient of the $\FI$-module $W(k)_j\otimes (\V^{\otimes k})_i$. Since $W(k)_j$ is a constant $\FI$-module and $(\V^{\otimes k})_i$ decomposes as a direct sum of summands of the form $\V_{l_1}\otimes\cdots\otimes\V_{l_k}$, such that $l_1+\cdots+l_k=i$, it follows by Proposition \ref{prop:weightstab-tensor} and Proposition \ref{prop:weightstab-subquotient} that $\weight(W(k)_j\otimes (\V^{\otimes k})_i)\le is$ and $\text{surj-deg}(W^j\otimes (\V^{\otimes k})_i)\le \max(r+s,is)$.

By the discussion above, we also have that $W(k)_j\otimes_{\Sigma_k} (\V^{\otimes k})_i$ is isomorphic to a direct sum of graded $\FI$-submodules of $(\V^{\otimes k}[j])_i$ (by decomposing $W(k)_j$ into irreducible $\Sigma_k$-representations and applying the corresponding Young symmetrizer for each summand), where $[j]$ denotes a shift of $j$ degrees upwards. Thus we get the same bound on injectivity degree, finishing the proof, since $i\le m$.
\end{proof}

\subsection{Derivation Lie algebras as $\FI$-Lie algebras} Now let us introduce more specific examples of $\FI$-modules that will be of interest to us. Here it will be useful to work with $\FI^\#$-modules. We make the following definition:

\begin{dfn}
Let $H$ be a graded vector space. We define a graded $\FI^\#$-module $\mathcal{H}$ by letting $\mathcal{H}(S):=H^{\oplus S}$, for any $S\in\FI$ and for any $A\subset S$, $B\subset T$ and bijection $f:A\to B$, we define a linear map $\mathcal{H}(f):\mathcal{H}(S)\to\mathcal{H}(T)$ as the composition 
$$\mathcal{H}(S)\twoheadrightarrow\mathcal{H}(A)\to\mathcal{H}(B)\hookrightarrow\mathcal{H}(T),$$
where the first map is the natural projection, the second is the map induced by $f$ and the last is the natural injection. \hfill $\diamond$
\end{dfn} 

In the following sections, $H$ will be the desuspension of the reduced homology of a simply connected finite CW-complex, so that its homology is finite dimensional. \iffalse The $\FI$-module $\mathcal{H}$ is clearly finitely generated, since $H$ is finite dimensional (in fact, it is generated by any basis of $\mathcal{H}(\mathbf{1})$), considered as an ordinary $\FI$-module. Furthermore, \fi We then have $\weight(\mathcal{H})=1$, since $H^{\oplus S}$ decomposes into a direct sum of trivial and standard representations of $\Sigma(S)$, which correspond to the padded partitions $\lambda[|S|]$ of $\lambda=(1)$ and $\lambda=(0)$, respectively. It is also easily verified that $\stabdeg(\mathcal{H})=1$.

Composing with the free graded Lie algebra functor $\L$, we get a new graded $\FI^\#$-module, which we denote by $\L \mathcal{H}$.

Since $\mathcal{H}$ is an $\FI^\#$-module, we may consider its dual $\FI^\#$-module $\mathcal{H}^*$. Let us describe it in some more detail. For a finite set $S$, we simply have $\HH^*(S)=\HH(S)^*=(H^*)^{\oplus S}$ and if $S\supseteq A\overset{f}{\to}B\subseteq T$ is a partial injection, then $\HH^*(S\supseteq A\overset{f}{\to}B\subseteq T)$ is the composition
$$\HH^*(S)\twoheadrightarrow \HH^*(A)\xrightarrow{ \circ H(f^{-1})}\HH^*(B)\hookrightarrow\HH^*(T).$$

\begin{rmk}
If we restrict this $\FI^\#$-module to $\FI$ and $i:S\hookrightarrow T$ is an injection, we can describe the map $\mathcal{H}^*(i)$ as follows: let $\phi\in\mathcal{H}^*(S)$ and $x_\alpha$ be in the summand of $H^{\oplus T}$ corresponding to $\alpha\in T$. Then
\begin{align}\label{H*-inclusion}
    \Big(\mathcal{H}^*(i)(\phi)\Big)(x_\alpha)=\begin{cases} 0&\text{ if }\alpha\in T\setminus i(S),\\ \left(\phi\circ\mathcal{H}(i)^{-1}\right)(x_\alpha)&\text{ if }\alpha\in i(S). \end{cases}
\end{align}
\end{rmk}

Just as for $\mathcal{H}$, the following proposition is easily verified:

\begin{prop}\label{prop-H*-finitely generated}
If $H$ is a finite dimensional graded vector space, then the graded $\FI^\#$-module  $\mathcal{H}^*$ has weight $\le 1$ and stability degree $\le 1$.\hfill $\diamond$
\end{prop}

Next, we will define the graded $\FI^\#$-Lie algebra of derivations on the graded $\FI^\#$-Lie algebra $\L\HH$. Recall that if $L$ is a graded Lie algebra, we define a derivation on $L$ as a (graded) linear map $D:L\to L$, which satisfies
$$D[x,y]=[Dx,y]+(-1)^{|x||D|}[x,Dy],$$
for all $x,y\in L$. We denote the graded vector space of all derivations by $\Der(L)$. With this, we make the following definition:

\begin{dfn}\label{dfn:DerFImodule}
    We define the graded $\FI$-module $\Der(\L\mathcal{H}):\FI\to\mathrm{grVect}_\Q$ by letting $\Der(\L\mathcal{H})(S)=\Der(\L\mathcal{H}(S))$ for $S\in \FI$, and for $i:S\hookrightarrow T$ an injection, we define $\Der(\L\mathcal{H})(i)$ as follows: Recall that a derivation on $\L\mathcal{H}(T)$ is uniquely determined by its restriction to $\mathcal{H}(T)$. Suppose therefore that $x_\alpha\in\mathcal{H}(T)$ lies in the direct summand of $\mathcal{H}(T)$ corresponding to $\alpha\in T$ and let $D\in \Der(\L (H^{\oplus S}))$. Then $\Der(\L\mathcal{H})(i)D$ is determined by
\begin{align}\label{DerLH-inclusion}
    \Big(\Der(\L\mathcal{H})(i)D\Big)(x_\alpha)=\begin{cases} 0 &\text{ if }\alpha\in T\setminus i(S),\\
    \left(\L\mathcal{H}(i)\circ D\circ \mathcal{H}(i)^{-1}\right)(x_\alpha) &\text{ if } \alpha\in i(S).\end{cases}
\end{align}
\end{dfn}

\begin{rmk}
    The functor $\Der(\L\HH)$ may in fact be extended to all of $\FI^\#$ using a similar definition, but since we will not be using this, we only consider the simpler functor from $\FI^\#$.
\end{rmk}

For a graded Lie algebra $L$, the commutator Lie bracket 
$$[D,D']=D\circ D'-(-1)^{|D||D'|}D'\circ D$$
makes $\Der(L)$ into a graded Lie algebra. A straightforward computation using \eqref{DerLH-inclusion} shows that
$$\Der(\L\mathcal{H})(i)[D,D'] =[\Der(\L\mathcal{H})(i)(D),\Der(\L\mathcal{H})(i)(D')] ,$$
giving us the following result:

\begin{prop}\label{prop:DerAsFImodule}
The functor $\Der(\L\HH):\FI\to\mathrm{grVect}_\Q$ of Definition \ref{dfn:DerFImodule} factors through the forgetful functor $\mathrm{grLie}_\Q\to\mathrm{grVect}_\Q$, where $\mathrm{grLie}_\Q$ is the category of graded Lie algebras over $\Q$.
\end{prop}

Furthermore, we can determine explicit weights and stability degrees in each degree of this graded $\FI$-module:

\begin{prop}\label{corollaryDerivations}
Let $H$ be a finite dimensional graded vector space concentrated in strictly positive degrees. If the degree of $H$ is bounded strictly below $d$, for some $d\ge 1$, we have $\weight\left(\Der(\L\mathcal{H})_m\right)\le m+d$ and $\stabdeg\left(\Der(\L\mathcal{H})_m\right)\le m+d$. 
\end{prop}

\begin{proof}\iffalse To prove this, we only need to work with the underlying $\FI$-module, which gives us less to verify.\fi For every $S\in\FI$, we have an isomorphism of graded vector spaces
$$\Psi_S:\mathcal{H}^*(S)\otimes\L\mathcal{H}(S)\overset{\cong}{\to}\Der(\L\mathcal{H})(S)$$
given by sending $\phi\otimes A\in \mathcal{H}^*(S)\otimes\L\mathcal{H}(S)$ to the derivation in $\Der(\L\mathcal{H})(S)$ defined by
$$x\mapsto \phi(x)A,$$
on $x\in\mathcal{H}(S)$. We want to prove that this defines a map of graded $\FI$-modules, i.e.\ that for every morphism $i:S\hookrightarrow T$, the diagram
\begin{equation}
    \begin{tikzcd}
        \mathcal{H}^*(S)\otimes\L\mathcal{H}(S)\arrow[r,"\Psi_S"]\arrow[d,"\mathcal{H}^*(i)\otimes\L\mathcal{H}(i)"']&\Der(\L\mathcal{H})(S)\arrow[d,"\Der(\L\mathcal{H})(i)"]\\
        \mathcal{H}^*(T)\otimes\L\mathcal{H}(T)\arrow[r,"\Psi_T"]&\Der(\L\mathcal{H})(T)
    \end{tikzcd}
\end{equation}
is commutative. This can be verified by applying the definitions of $\Psi_S$ and $\Psi_T$, together with the description of $\mathcal{H}^*(i)$ given by (\ref{H*-inclusion}), and the description of $\Der(\L\mathcal{H})(i)$, given by (\ref{DerLH-inclusion}).

Thus $\Der(\L\mathcal{H})\cong\mathcal{H}^*\otimes\L\mathcal{H},$ as graded $\FI$-modules. Note that $\mathcal{H}^*$ is concentrated in \textit{negative} degrees, but which are bounded from below, by the assumption on $H$. We thus have $$(\mathcal{H}^*\otimes\L\mathcal{H})_m= \bigoplus_{i=1}^{d-1}(\mathcal{H}^*)_{-i}\otimes (\L\mathcal{H})_{m+i}.$$
Since both $\mathcal{H}$ and $\mathcal{H}^*$ are of weight and stability degree 1, the same argument as in the proof of Proposition \ref{propSchurFT} shows that $(\mathcal{H}^*)_{-i}\otimes (\mathbb{L}\mathcal{H})_{m+i}$ is simultaneously a quotient of an $\FI$-module of weight and stability degree $\le m+i+1$, so $(\mathcal{H}^*\otimes\L\mathcal{H})_m$ thus has both weight and stability degree $m+d$, due to Propositions \ref{prop:weightstab-subquotient} and \ref{cor:stabdeg-summand}.
 \end{proof}

The $\FI$-modules that we consider in Theorem \ref{theorem1} are the homology groups of graded $\FI$-modules of the type $\Der(\L\mathcal{H})$, with $H$ as above, so it will follow immediately from Proposition \ref{prop:weightstab-homology} that we get the claimed bounds on weight and stability degree. In the case of Theorem \ref{theorem2}, it turns out that we can use Proposition \ref{propSchurFT} more directly, due to results from \cite{BM14}.

\subsection{$\FI$-Lie models}\label{sec:FI-models} Now, let us introduce the notion of an $\FI$-\textit{Lie model}, which will be one of the main tools of the paper. For the basic theory of Lie models in rational homotopy theory, see for example \cite{felixrht}.

\begin{dfn}\label{Definition-FI-model} Let $\mathcal{X}$ be a simply connected based $\FI$-space and let $\mathcal{L}$ be a dg $\FI$-Lie algebra.  We say that $\mathcal{L}$ is an $\FI$-Lie model for $\mathcal{X}$ if\begin{enumerate}
    \item for every $S\in\FI$, $\mathcal{L}(S)$ is a dg Lie model for the space $\mathcal{X}(S)$ and
    \item for every morphism $S\hookrightarrow T$ in $\FI$, the dgl map
    $$\mathcal{L}(S)\to\mathcal{L}(T)$$
    is a model for the map $\mathcal{X}(S)\to\mathcal{X}(T)$.\hfill$\diamond$
\end{enumerate}

\end{dfn}

\begin{rmk}
If $\mathcal{L}$ is an $\FI$-Lie model for $\mathcal{X}$, then $H_*(\mathcal{L})\cong\pi_*^\Q(\mathcal{X})$ is an isomorphism of $\FI$-modules. \strut\hfill$\diamond$
\end{rmk}

\begin{rmk}A reader may feel that Definition \ref{Definition-FI-model} is somewhat unnatural. 
%Definition \ref{Definition-FI-model} works for our purposes, but 
Indeed, it is not the ``philosophically'' correct definition of $\FI$-Lie model, seen from a modern homotopy theoretic perspective. There is an equivalence of $\infty$-categories
$$(\mathrm{dgLie}_\Q)_{\ge 1}\cong\mathrm{Top}_{\ge 2}^\Q,$$
between the $\infty$-categories of connected dg Lie algebras, localized at the quasi-isomorphisms, and simply connected spaces, localized at the rational homotopy equivalences. The usual definition of dg Lie models in rational homotopy theory may be expressed by saying that a connected dg Lie algebra $(L,d)$ is a dg Lie model for a simply connected space $X$ if they are isomorphic under this equivalence. Equivalently, it suffices to require that they are isomorphic under the equivalence between the homotopy categories $h(\mathrm{dgLie}_\Q)_{\ge 1}\cong h\mathrm{Top}^\Q_{\ge 2}$. The correct definition of $\FI$-Lie model should therefore be that a dg $\FI$-Lie algebra $\mathcal{L}$ is a $\FI$-Lie model of a simply connected pointed $\FI$-space $\mathcal{X}$ if they are isomorphic under the equivalence of the homotopy categories
$$h\mathrm{Fun}(\FI,(\mathrm{dgLie}_\Q)_{\ge 1})\cong h\mathrm{Fun}(\FI,\mathrm{Top}^\Q_{\ge 2}).$$
In contrast, our definition is requiring isomorphism under the equivalence of ``ordinary'' functor categories
$$\mathrm{Fun}(\FI,h(\mathrm{dgLie}_\Q)_{\ge 1})\cong\mathrm{Fun}(\FI,h\mathrm{Top}^\Q_{\ge 2}).$$
Nevertheless, the naive Definition \ref{Definition-FI-model} is simpler and sufficient for the purposes of this paper.\hfill $\diamond$
\end{rmk}

\section{Rational homotopy theory for homotopy automorphisms}
In this section, we will review some rational homotopy theory for homotopy automorphisms needed for this paper.

Let $X$ be a simply connected topological space homotopy equivalent to a CW-complex. A homotopy automorphism of $X$ is a self-map $\varphi\colon X\to X$, that is a homotopy equivalence. We denote the topological monoid of unpointed and pointed homotopy automorphisms of $X$ by $\aut(X)$ and $\aut_*(X)$, respectively. Given a subspace $A\subset X$, we denote the topological monoid of $A$-relative homotopy automorphisms of $X$, i.e.  the homotopy automorphisms that preserve $A$ pointwise, by $\aut_A(X)$. When $A$ is a point or empty we simply write $\aut_*(X)$ and $\aut(X)$, respectively and when $X=N$ is a manifold with boundary $A=\partial N$, the monoid of boundary relative homotopy automorphisms of $N$ is denoted by $\aut_\partial(N)$.

 If $X$ is well pointed and $A\subset X$ is a cofibration of cofibrant spaces in the Hurewicz model structure, then all of $\aut(X)$, $\aut_*(X)$ and $\aut_A(X)$ are group-like monoids, and thus equivalent to topological groups.  We take the basepoint of a topological monoid $G$ to be the identity element and $\pi_k(G,\id)$ is abbreviated by $\pi_k(G)$.  We  denote the classifying space of $G$ by $BG$ and its universal cover by $\widetilde{BG}$. Moreover, if a topological monoid $G$ is group-like, then $G$ and $\Omega BG$ are weakly equivalent as topological monoids.  Let  $G_\circ\subset G$ denote the connected component of the identity. We have that $BG_\circ \simeq \widetilde{BG}$. We observe that 
$$\pi_k(G)\ot\Q\cong \pi_{k+1}(\widetilde{BG})\ot\Q\cong \pi_{k+1}(BG_\circ)\ot\Q\cong H_k(\mathfrak g_{BG_\circ})$$
for all $k\geq 1$ and where $\mathfrak g_{BG_\circ}$ is any dg Lie algebra model for $BG_\circ$.

The identity component of $\aut_A(X)$ is denoted by $\aut_{A,\circ}(X)$.
\begin{rmk}
By \cite{farjoun}, there are functorial and continuous rationalization
functors that preserve cofibrations. In particular, given a cofibration $A\subset X$, there is a rationalization functor that induces  a group homomorphism $r\colon\pi_0(\aut_A(X))\to \pi_0(\aut_{A_\Q}(X_\Q))$.\hfill$\diamond$\end{rmk}

 For $k\geq 1$ we have that
$$
\pi_k(\aut_A(X))\ot\Q \cong \pi_k(\aut_{A_\Q}(X_\Q)),
$$
since $B\aut_{A,\circ}(X)_\Q \simeq B\aut_{A_\Q,\circ}(X_\Q)$, see \cite[Proposition 2.4]{BS19}.

A model for $B\aut_{A,\circ}(X)$ is given in terms of dg Lie algebras of derivations. 

\begin{dfn}
Given a dg Lie algebra $L$, let $\Der(L)$ denote the dg Lie algebra of derivations of $L$ where the graded Lie bracket is given by 
$$[\theta,\eta] = \theta\circ \eta - (-1)^{|\theta||\eta|}\eta\circ\theta,$$ and the differential is given by $\partial = [d_L,-]$ where $d_L$ is the differential of $L$.\hfill$\diamond$
\end{dfn}

\begin{dfn}
Given a chain complex  $C = C_*$, the positive truncation of $C$, denoted by $C^+$, is given by
$$C^+_i =\left\{\begin{array}{ll}
C_i&\quad\text{if }i>1\\
\ker(C_1\xrightarrow dC_0)&\quad\text{if }i=1\\
0&\quad\text{if }i<1.
\end{array}\right.$$\hfill$\diamond$
\end{dfn}

\begin{dfn}
A dg Lie algebra $(\L(V),d)$ is called  \textit{quasi-free} if its underlying graded Lie algebra structure is a free graded Lie algebra on graded vector space $V$.
\end{dfn}
\begin{dfn}
We say that a dg Lie algebra map between two quasi-free dg Lie algebras $\phi\colon \L(V)\to \L(U)$ is free if $\phi$ is injective and $\phi(V)\subseteq U$. In particular $U$ has a decomposition $U \cong V\oplus W$. 
\end{dfn}

\begin{rmk}
One can show that the free maps between the quasi-free dg Lie algebras are exactly the cofibrant maps between them (see the remark after \cite[Proposition 5.5]{quillen69}).\end{rmk}

 \begin{prop} \label{prop:models-for-haut}
\begin{itemize} 
\item[(a)] Let $X$ be a simply connected space of the homotopy type of a finite CW-complex with a quasi-free dg Lie algebra model $\L_X$. A dg Lie model for $B\aut_{*,\circ}(X)$ is given by $\Der^+(\L_X)$.
\item[(b)] Let $A\subset X$ be cofibration of simply connected spaces of the homotopy type of finite CW-complexes, and let $\L_A\to\L_X$ be a cofibration (i.e a free map) of quasi-free dg Lie algebras that models the inclusion $A\subset X$. A dg Lie model for  $B\aut_{A,\circ}(X)$ is given by the positive truncation of the dg Lie algebra of derivations on $\L_X$  that vanish on $\L_A$, denoted by $\Der^+(\L_X\|\L_A)$.
\item[(c)] The inclusion $\Der^+(\L_X\|\L_A) \to \Der^+(\L_X)$ is a model for $B\aut_{A,\circ}(X)\to B\aut_{*,\circ}(X)$ induced by the inclusion $\aut_{A,\circ}(X)\hookrightarrow\aut_{*,\circ}(X)$.
\hfill $\diamond$
\end{itemize}
\end{prop}
\begin{proof}
For (a), see \cite[Corollarie VII.4. (4)]{tanre83}. For (b), see \cite[Theorem 1.1]{BS19}. Statement (c) follows by \cite[Proposition 4.6]{BS19} and the theory established in \cite[§ 3.4, 3.5]{berglund17}.
\end{proof}

We recall the notion of  geometric realizations of dg Lie algebras. For a detailed account on the subject refer the reader to \cite{hinich97}, \cite{getzler09}, \cite{berglund15} and \cite{berglund17}.  
\begin{dfn}[\text{\cite[Definition 2.1.1]{hinich97}}]
Let $\Omega_\bullet=\Omega_\bullet^*$ denote the simplicial commutative dg algebra in which $\Omega_n^*$ is the Sullivan-de Rham algebra of polynomial differential forms on the $n$-simplex. The geometric realization of a positively graded dg Lie algebra $L$   is defined to be the simplicial set $\MC(L\ot\Omega_\bullet)$ of Maurer-Cartan elements of the simplicial dg Lie algebra $L\ot\Omega_\bullet$, denoted by $\MC_\bullet(L)$. We recall that the tensor product $L\ot \Omega$ of a dg Lie algebra $L$ with a commutative dg algebra $\Omega$ is again a dg Lie algebra, where $[\ell_1\ot c_1, \ell_2\ot c_2] = (-1)^{|c_1||l_2|}[\ell_1,\ell_2]\ot c_1c_2$.
 A positively graded dg Lie algebra $L$ is a Lie model for a simply connected space $X$ if and only if there exists a zig-zag of rational homotopy equivalences
between the geometric realization $\MC_\bullet(L)$ and $X$.\hfill$\diamond$ \end{dfn}

 The functor $\MC_\bullet$ takes surjections to Kan fibrations (\cite[Proposition 4.7]{getzler09}) and takes injections to cofibrations (in the classical model structure on simplicial sets). In particular if $\L_A\to \L_X$ is a free map of dg Lie algebras that models a cofibration $A\subset X$, then the cofibration $\MC_\bullet(\L_A)\hookrightarrow \MC_\bullet(\L_X)$ is a simplicial model for the cofibration $A_\Q\subset X_\Q$. Thus,  $\aut_{A_\Q}(X_\Q)$ and $\aut_{\MC_\bullet(\L_A)}(\MC_\bullet(\L_X))$ are weakly equivalent as topological monoids.

\begin{dfn}
The exponential $\exp(\mathfrak h)$ of a nilpotent Lie algebra  $\mathfrak h$ concentrated in degree zero is the nilpotent group with the underlying set given by $\mathfrak h$ and with multiplication given by the Baker-Campbell-Hausdorff formula. The exponential of a positively graded dg Lie algebra $L$, denoted $\exp_\bullet(L)$, is the simplicial group given by the exponential  $\exp(Z_0(L\ot \Omega_\bullet))$ of the zero cycles in $L\ot \Omega_\bullet$ (see \cite{berglund17}).\hfill $\diamond$
\end{dfn}

\begin{prop}[\text{\cite[Corollary 3.10]{berglund17}}]
For a positively graded dg Lie algebra $L$, there is an equivalence of topological monoids between $\exp_\bullet(L)$ and the loop space
$\Omega\MC_\bullet(L)$.\hfill $\diamond$
\end{prop}

\begin{dfn}
Let $ \L(V)\subset \L(V\oplus W)$ be a cofibration of free positively graded dg Lie algebras and let $\Der(\L(V\oplus W)\|\L(V))$ denote the dg Lie algebra of derivations on $\L(V\oplus W)$ that vanish on $\L(V)$ (the differential is $[d_{\L(V\oplus W)},-])$.
There is a left action of $\exp_\bullet(\Der^+(\L(V\oplus W)\|\L(V)))$ on $\MC_\bullet(\L(V\oplus W))$ given by 
\begin{equation}\label{eqn:monoid-equivalence}
\Theta.x= \sum_{i\geq 0} \frac{\Theta^i(x)}{i!}
\end{equation}
(see \cite[§ 3.2]{BS19}).\hfill$\diamond$
\end{dfn}

\begin{prop}[cf. \text{\cite[Proposition 3.7]{berglund20}}]\label{prop:equivalence-of-monoids-two}
Let $A \subset X$ be a cofibration of simply connected
spaces with homotopy types of finite CW-complexes, and let $\iota\colon \L(V) \to \L(V\oplus W)$ be a  free map of quasi-free dg Lie algebras that models the inclusion $A \subset X$. Then the topological monoid map
$$F\colon \exp_\bullet(\Der^+(\L(V\oplus W)\|\L(V)))\to\aut_{\MC_\bullet(\L(V)),\circ}(\MC_\bullet(\L(V\oplus W)))\simeq\aut_{A_\Q,\circ}(X_\Q),$$
$$F(\Theta)(x)=\Theta.x$$
is a weak equivalence.
\end{prop}

\newcommand{\h}{\mathfrak h}
\begin{proof}
 Note that the action of $\exp_\bullet(\Der^+(\L(V\oplus W)\|\L(V)))$ on $\MC_\bullet(\L(V\oplus W))$ fixes $\MC_\bullet(\L(V))\subset \MC_\bullet(\L(V\oplus W))$ pointwise. In particular, the group action yields a map $$\exp_\bullet(\Der^+(\L(V\oplus W)\|\L(V))) \to \aut_{\MC_\bullet(\L(V))}(\MC_\bullet(\L(V\oplus W))).$$ Moreover, since $\exp_\bullet(\Der^+(\L(V\oplus W)\|\L(V)))$ is connected and $F$ preserves the identity element, we may replace the codomain by $$\aut_{\MC_\bullet(\L(V)),\circ}(\MC_\bullet(\L(V\oplus W)))$$ (i.e.\ the component of the identity).
We proceed by following the proof of \cite[Proposition 3.7]{berglund20}, and adapting it to our situation.
Given a positively graded dg Lie algebra $\h$, there is an isomorphism of abelian groups 
$$
G\colon H_k(\h)\to \pi_k(\exp_\bullet(\h))
$$
where a homology class of a cycle $z\in Z_k(\h)$ is sent to the homotopy class of the $k$-simplex $ z\ot\nu_k \in Z_0( \h\ot \Omega^*_k)$ where $\nu_k$ is the class $k!dt_1\cdots dt_k$. That $G$ defines an isomorphism is motivated in the proof of \cite[Proposition 3.7]{berglund20}. 

We have that $\nu_k^2 = 0$, and consequently 
$$F(\theta\ot\nu_k) = \id+\theta\ot\nu_k.$$

Let us now  analyze $\pi_k(\aut_{\MC_\bullet(\L(V)),\circ}(\MC_\bullet(\L(V\oplus W))))$, $k\geq 1$, as in the proof of \cite[Theorem 3.6]{BM14}. In order to simplify notation, $\MC_\bullet(\L(V))$ is denoted by $A_\Q$ and $\MC_\bullet(\L(V\oplus W))$ by $X_\Q$. We have that an element  $f\in\pi_k(\aut_{A_\Q,\circ}(X_\Q))$ is represented by a map  
$$ f\colon  (S^k\sqcup*)\wedge X_\Q\to X_\Q$$ where $f(*,x)=x$ for every $x\in X_\Q$ and  and $f(s,a)=a$ for every $a\in A_\Q$ and $s\in S^k$.

A dg Lie algebra model for $ (S^k\sqcup*)\wedge X_\Q$ is given by  $(\L(U\oplus s^kU),\partial)$ where $U= V\oplus W$ and with a differential determined by the following: Let $d$ be the differential on $\L(U)$, then $\partial(u) = d(u)$ for every $u\in U$ and $\partial(s^ku)= (-1)^ks^kd(u)$ for every $s^ku\in s^kU$.

Now, $f\colon  (S^k\sqcup*)\wedge X_\Q\to X_\Q$ is modelled  by some map $\varphi_f\colon \L(U\oplus s^k U)\to \L(U)$ that satisfy $\varphi_f(u)=u$ for every $u\in U$ and $\varphi_f(s^kv)=0$ for every $v\in V\subset U$. Now, let $\theta_f$ be the unique derivation on $\L(U)$ that satisfy $\theta_f(u)= \varphi_f(s^k u)$ for every $u\in U$. Note that $\theta_f$ is a cycle and that it vanishes on $\L(V)$, i.e.\ $\theta_f\in Z_k(\Der^+(\L(U)\|\L(V)))$. Also note that if $f= \pi_k(F)[\theta\ot \nu_k]$ then $\theta_f =\theta$. Let $K\colon   
\pi_k(\aut_{A_\Q,\circ}(X_\Q))\to H_k(\Der^+(\L(U)\|\L(V)))$ be given by $K(f)=\theta_f$. It follows from \cite{BM14} and \cite{LuptonSmith} that  this map is well-defined and is an isomorphism.

Set $\h=\Der^+(\L(U)\|\L(V))$. The composition
$$
H_k(\h)\xrightarrow G \pi_k(\exp_\bullet(\h))\xrightarrow{\pi_k(F)} \pi_k(\aut_{A_\Q,\circ}(X_\Q))\xrightarrow K H_k(\h)
$$
is the identity map, which forces $\pi_k(F)$ to be an isomorphism. This proves that $F$ is a weak equivalence.
\end{proof}

We recall the following: A topological group $G'$ acts on itself by conjugation $G'\to \Aut(G')$, $g\mapsto \kappa_g$ where $\kappa_g(h) = ghg^{-1}$. If $g$ and $g'$ belong to the same connected component of $G'$, then $\kappa_g$ and $\kappa_{g'}$ are homotopic and this induce equal maps on the homotopy groups of $G'$.  This group action restricts to a group action on the identity component $G'_\circ$ of $G'$, which in turn induces an action of $G'$ on $BG'_\circ$. This gives that $\pi_0(G')$ acts on $\pi_*(BG'_\circ)$. Since a group-like monoid $G$ is equivalent to a topological group $G'$, we have that $\pi_0(G)$ acts on $\pi_*(G_\circ)$.

In the rest of this section discuss the action of $\pi_0(\aut_A(X))$ on $\pi_k(\aut_A(X))$ from a rational homotopy point of view. In order to do that we need to recall some of the theory established in \cite{ES20}.

\begin{prop}[\text{\cite[Theorem 1.3]{ES20}}]
Given a map $f\colon \L(V)\to \mathfrak g$ of positively graded dg Lie algebras, there exists a minimal relative model $q\colon\L(V\oplus W)\xrightarrow\sim\mathfrak g$ for $f$ in the following sense:
 \begin{itemize}
 \item[{\normalfont (a)}] $\L(V)$ is a dg subalgebra of $\L(V\oplus W)$ and $f = q\circ \iota$, where $\iota\colon \L(V)\to \L(V\oplus W)$ is the inclusion.
 \item[{\normalfont(b)}] Given a quasi-isomorphism $g\colon \L(V\oplus W)\to \L(V\oplus W)$, where $g$ restricts to an automorphism of $\L(V)$, then $g$ is an automorphism.\hfill $\diamond$
 \end{itemize}
\end{prop}

\begin{dfn}
Let $\iota\colon\L(V)\to \L(V\oplus W)$ be a free map of quasi-free dg Lie algebras. We say that an endomorphism  $\varphi\colon\L(V\oplus W)\to \L(V\oplus W)$ is $\iota$-relative if $\varphi|_{\L(V)}= \id$. 
We say that two $\iota$-relative endomorphisms $\varphi$ and $\psi$ are $\iota$-equivalent if there exists a homotopy $h \colon \L(V\oplus W) \to  \L(V\oplus W)\ot \Lambda(t,dt)$ from $\varphi$ to $\psi$ that preserves $\L(V)$
in the following sense: $h(v) = v \ot1$ for every $v\in\L(V)$ (cf. \cite[§ 14 (a)]{felixrht}).

We denote the group of $\iota$-relative automorphisms of $\L(V\oplus W)$ by $\Aut_\iota(\L(V\oplus W))$.
\hfill$\diamond$
\end{dfn}

\begin{lemma}[\text{\cite[Corollary 4.6]{ES20}}]\label{lemma:espicsaleh}
Let $\iota\colon\L(V)\to \L(V\oplus W)$ be a minimal relative dg Lie model  for a cofibration  $A\subset X$ of simply connected spaces. Then there are  group isomorphisms  
$$\Aut_\iota(\L(V\oplus W))/\iota\text{-equivalence}\cong 
\pi_0(\aut_{\MC_\bullet(\L(V))}(\MC_\bullet(\L(V\oplus W))))
\cong\pi_0(\aut_{A_\Q}(X_\Q)).$$
\ \ \ \hfill $\diamond$
\end{lemma}

\begin{rmk}\label{rmk:espicsaleh}
By this lemma, it makes sense to refer to an $\iota$-relative automorphisms of $\LVW$ as an algebraic model for an $A_\Q$-relative homotopy automorphism of $X_\Q$. \hfill $\diamond$
\end{rmk}

\begin{dfn}\label{dfn:Aut-action}
Consider the group action of  $\Aut_\iota(\L(V\oplus W)$ on $\Der(\L(V\oplus W)\|\L(V))$ given by the following: For $\varphi\in \Aut_\iota(\L(V\oplus W))$ and $\theta\in \Der(\L(V\oplus W)\|\L(V))$, let
$$
\varphi.\theta = \varphi\circ \theta\circ \varphi^{-1}.
$$
This induces an action of $\Aut_\iota(\L(V\oplus W))$ on $\exp_\bullet(\Der(\L(V\oplus W)\|\L(V)))$.

There is also an action of $\Aut_\iota(\L(V\oplus W))$ on  $\aut_{\MC_\bullet(\L(V)),\circ}(\MC_\bullet(\L(V\oplus W)))$; for $\varphi\in \Aut_\iota(\L(V\oplus W))$ and $f\in \aut_{\MC_\bullet(\L(V)),\circ}(\MC_\bullet(\L(V\oplus W)))$, let 
$$
\varphi.f = \MC_\bullet(\varphi)\circ f \circ \MC_\bullet(\varphi^{-1}). 
$$
\hfill$\diamond$
\end{dfn}

\begin{prop}\label{prop:F-is-equivariant}
The equivalence
$$F\colon \exp_\bullet(\Der^+(\L(V\oplus W)\|\L(V)))\to\aut_{\MC_\bullet(\L(V)),\circ}(\MC_\bullet(\L(V\oplus W)))\simeq\aut_{A_\Q,\circ}(X_\Q),$$
of Proposition \ref{prop:equivalence-of-monoids-two} is $\Aut_\iota(\LVW)$-equivariant with respect to the actions in Definition \ref{dfn:Aut-action}.
\end{prop}

\begin{proof}
This is a straightforward verification left to the reader.
\end{proof}

\begin{cor}\label{cor:F-is-equivariant}
Let $f\in\aut_{A_\Q}(X_Q)$ and let $\varphi\in\Aut_\iota(\LVW)$ be a relative model for $f$. The automorphism 
$$\alpha_\varphi\colon\Der(\LVW\|\L(V))\to \Der(\LVW\|\L(V))$$
$$
\alpha_\varphi(\theta)=\varphi\circ\theta\circ \varphi^{-1}$$
is a model for the delooping of the homotopy autmorphism
$$    \mathrm{Ad}_f\colon \aut_{A_\Q}(X_\Q)\to \aut_{A_\Q}(X_\Q)$$
$$    \mathrm{Ad}_f(g) = f\circ g\circ f^{-1},$$
where $f^{-1}$ is an $A_\Q$-relative homotopy inverse to $f$.\hfill $\diamond$
\end{cor}

\section{Homotopy automorphisms of wedge sums}\label{sec:wedge}
We fix some  notation for this section. Let $(X,*)$ be a fixed  simply connected space homotopy equivalent to a finite CW-complex. For any finite set $S$, let $X_S := \bigvee^S X$. For any morphism $S\hookrightarrow T$ in $\FI$, there is an obvious induced base point preserving map $X_S\hookrightarrow X_T$ given by inclusion of wedge summands in the order specified by the injection $S\hookrightarrow T$. Thus the functor $S\mapsto X_S$ is a pointed $\FI$-space, which we will denote by $\X$. 

We fix a quasi-free dg Lie algebra model $\L(H)=(\L(H),d_{\L(H)})$ for $X$. A dg Lie model for $X_S$ is given by the $S$-fold free product of dg Lie algebras
$$
\L(H)^{*S}:= \L(H)*\dots*\L(H) \cong \L(H^{\oplus S})
$$
(see \cite[§ 24 (f)]{felixrht}). The association $S\mapsto \L(H^{\oplus S})$ defines a dg $\FI$-Lie algebra $\L\mathcal H$. Given a morphism $i:S\hookrightarrow T$ in $\FI$, we get an induced inclusion $H^{\oplus S}\hookrightarrow H^{\oplus T}$, which induces an inclusion $\L(H^{\oplus S})\hookrightarrow\L(H^{\oplus T})$ that models the map $\X(i):\X(S)\to\X(T)$ (this follows from e.g. Example 1 in § 12 (c) and Example 2 in § 24 (f) in \cite{felixrht}). Thus $S\mapsto \L\left(H^{\oplus S}\right)$ defines a dg $\FI$-Lie model for the pointed FI-space $S\mapsto \X(S)$.

We proceed and define another pointed $\FI$-space $\aut_*(\mathcal X)$ (the base point is always the identity) as follows: for $S\in\FI$, we let $\aut_*(\mathcal X)(S):= \aut_*( X_S)$. For $i:S\hookrightarrow T$ in $\FI$, we get a map $\aut_*(X_S)\hookrightarrow\aut_*(X_T)$, defined as follows: suppose $x_\alpha\in X_T$ lies in the wedge summand of $X_T$ corresponding to $\alpha\in T$ and let $f\in \aut_*(X_S)$. Then
\begin{align*}\label{aut-inclusion}
    \Big(\aut_*(\mathcal X)(i)f\Big)(x_\alpha)=\begin{cases} x_\alpha &\text{ if }\alpha\in T\setminus i(S),\\
    \left(\mathcal{X}(i)\circ f\circ \mathcal{X}(i)^{-1}\right)(x_\alpha) &\text{ if } \alpha\in i(S).\end{cases}
\end{align*}
Note that $\aut_*(\mathcal X)(i)f$ is in some sense an extension by identity of $f$. For instance, if $i_s\colon \mathbf n\to \mathbf{n+1}$ is the standard inclusion, then $\aut_*(\mathcal X)(i_s)f$ is the homotopy automorphism of $X_{\mathbf{n+1}}$ that coincides with $f$ on the first $n$ wedge summands, and is the identity on the last summand.

Restricting to the identity component gives a pointed sub-$\FI$-space $\aut_{*,\circ}(\mathcal X)$.
We are interested in the rational homotopy groups of this $\FI$-space.

\begin{rmk}\label{rmk:htpy-grps-of-htpy-aut}
Now it is tempting to say that we will construct an $\FI$-Lie model for $\aut_*(\X)$. However, this pointed $\FI$-space is generally not simply connected. Instead, we take a functorial classifying space construction $B\colon \mathrm{TopMon}\to \mathrm{Top}_*$ from the category of topological monoids to the category of pointed topological spaces, and consider the pointed $\FI$-space  $B\aut_{*,\circ}(\X)$, where $\aut_{*,\circ}(X_S)$ is the identity component of $\aut_*(X_S)$, for every $S\in\FI$. For every $S$, we have $B\aut_{*,\circ}(X_S)\simeq \widetilde{B\aut_*(X_S)}$ and so this is a simply connected pointed $\FI$-space, which enables us to apply our tools from rational homotopy theory. Furthermore, for every $k\ge 1$, we have 
$$\pi_k^\Q(\aut_*(X_S))\cong\pi_{k+1}^\Q(B\aut_{*,\circ}(X_S)),$$
so 
$$\pi_k^\Q(\aut_
*(\X))\cong\pi_{k+1}^\Q(B\aut_{*,\circ}(\X)),$$
as $\FI$-modules.\hfill $\diamond$
\end{rmk}

We have by Proposition \ref{prop:models-for-haut} (a) that a model for $B\aut_{*,\circ}(X_S)$ is given by $\Der^+(\L(H^{\oplus S}))$, with differential given by $[d_{\L(H^{\oplus S})},-]$. 
The inclusion $\L\mathcal H(i)\colon \L(H^{\oplus S})\hookrightarrow\L(H^{\oplus T})$ induces a graded Lie algebra $\Der^+(\L(H^{\oplus S}))\hookrightarrow \Der^+(\L(H^{\oplus T}))$ map as discussed in \ref{prop:DerAsFImodule}. Moreover, this map commutes with the differential, i.e.\ a dg Lie algebra map. This together with Proposition \ref{prop:DerAsFImodule}, yields that we have a dg $\FI$-Lie algebra $(\Der^+(\L\mathcal H),[d_{\L\mathcal H},-])$.
  
We will show that $(\Der^+(\L\mathcal H),[d_{\L\mathcal H},-])$ defines an $\FI$-Lie model for the pointed $\FI$-space $B\aut_{*,\circ}(\mathcal X)$.

\begin{prop}\label{prop:models-for-stab-sigma}
 \begin{itemize}
     \item[(a)] Let $i_s\colon \mathbf n\to \mathbf{n+1}$ denote the standard inclusion. Then a dg Lie algebra model for $$B\aut_{*,\circ}(\X)(i_s)\colon B\aut_{*,\circ}(X_n)\to B\aut_{*,\circ}(X_{n+1})$$ is given by 
     $$\varphi_n:=
     \Der^+(\L\mathcal H)(i_s)\colon\Der^+(\L(H^{\oplus n}))\to \Der^+(\L(H^{\oplus n+1}))
     $$
     
     \item[(b)] The $\Sigma_n$-action on $B\aut_{*,\circ}(X_n)$ is modelled by the $\Sigma_n$-action on $\Der^+(\L(H^{\oplus n}))$.
 \end{itemize}
\end{prop}

\begin{proof}
(a):
In order to simplify notation, let $\L(H^{\oplus k})$ be denoted by $\L_k$, and let $H_\ell\cong H$ denote the last summand of $H^{\oplus n+1}$, and let $\L_\ell$ denote $\L(H_\ell)$. In particular we have that $\L_{n+1} = \L_n*\L_\ell$. Let $c_n\colon \MC_\bullet(\L_n)\to \MC_\bullet(\L_{n+1})$ and $c_l\colon \MC_\bullet(\L_\ell)\to \MC_\bullet(\L_{n+1})$ denote the cofibrations induced by the standard inclusion $\L_n\to \L_n*\L_\ell$ and $\L_\ell\to \L_n*\L_\ell$ respectively.

From  Proposition \ref{prop:equivalence-of-monoids-two} we get topological monoid equivalences
$$F_n\colon \exp_\bullet(\Der^+(\L_n))\to\aut_*(\MC_\bullet(\L_n)).$$
Those maps have adjoints 
$$\tilde F_n\colon \exp_\bullet(\Der^+(\L_n))\times\MC_\bullet(\L_n)\to \MC_\bullet(\L_n).$$
We have by the explicit formulas for $\{F_n\}$ that
$$\tilde F_{n+1}\circ(\exp_\bullet(\varphi_n)\times c_n) = c_n\circ \tilde F_n.$$
In particular, $F_{n+1}\circ \exp_\bullet(\varphi_n)(\Theta)$ is an extension of $F_n(\Theta)$, for $\Theta\in \exp_\bullet(\Der^+(\L_n))$.

We also have that $$\tilde F_{n+1}\circ(\exp_\bullet(\varphi_n)\times c_\ell)(g,x) = c_\ell(x),\quad \forall (g,x)\in \exp_\bullet(\Der^+(\L_n))\times \MC_\bullet(\L_\ell).$$
In particular $F_{n+1}\circ \exp_\bullet(\varphi_n)(\Theta)$ restricts to the identity on $\MC_\bullet(\L_\ell)\subset \MC_\bullet(\L_{n+1})$.
That means that $\exp_\bullet(\varphi_n)$ is a simplicial model for $\aut_{*,\circ}(\X)(i_s)\colon \aut_{*,\circ}(X_n)\to \aut_{*,\circ}(X_{n+1})$. This gives (a).

(b): This is a direct consequence of Proposition \ref{prop:F-is-equivariant}  and Corollary \ref{cor:F-is-equivariant}.\end{proof}

\begin{thm} \label{thm:model-for-stab-action-one}  $(\Der^+(\L\mathcal H),[d_{\L\mathcal H},-])$ is an $\FI$-Lie model for the pointed $\FI$-space $B\aut_{*,\circ}(\X)$.
\end{thm}
\begin{proof}
By the second part of Lemma \ref{lemma:consistentSeqFromFI-module}, an $\FI$-module is completely determined its underlying consistent sequence. By Proposition \ref{prop:models-for-stab-sigma} the stabilization maps and the $\Sigma_n$-actions defining the consistent sequence for the dg $\FI$-Lie algebra $(\Der^+(\L\mathcal H),[d_{\L\mathcal H},-])$ models the stabilization maps and the $\Sigma_n$-actions defining the consistent sequence for the $\FI$-space $B\aut_{*,\circ}(\X)$. From this we conclude that  $(\Der^+(\L\mathcal H),[d_{\L\mathcal H},-])$ is an $\FI$-Lie model for the pointed $\FI$-space $B\aut_{*,\circ}(\X)$.
\end{proof}

We have now all ingredients needed for proving Theorem \ref{theorem1}.

\theoremone*

\begin{proof}
We will use the established terminology in this section.
We have already seen in Theorem \ref{thm:model-for-stab-action-one} that $(\Der^+(\L\mathcal H),[d_{\L\mathcal H},-])$ is an $\FI$-Lie model for $B\aut_{*,\circ}(\X)$. Since $H_k(\Der^+(\L\mathcal H))\cong \pi_k^\Q(\aut_*(\X))$ (see Remark \ref{rmk:htpy-grps-of-htpy-aut}) it is thus enough to prove that $H_k(\Der^+(\L\mathcal H))$ has the stated bounds on weight and stability degree.

Since $(\Der^+(\L\mathcal{H}),[d_{\L\mathcal H},-])$, defines a dg $\FI$-Lie algebra, it follows that $H_*(\Der^+(\L\mathcal{H}))$
is a graded $\FI$-module. The truncation is defined precisely so that $H_k(\Der^+(\L\mathcal{H}))\cong H_k(\Der(\L\mathcal{H}))$ for all $k\ge 0$. Since $H=s^{-1}\tilde{H}_*(X)$ and $X$ is assumed to be simply connected, $H$ is finite dimensional and concentrated in positive degree and since we have assumed that the homology of $X$ vanishes in degree at least $d$, $H$ is concentrated in degrees strictly below $d-1$. The given bounds on stability degree and weight now follow from Propositions \ref{corollaryDerivations} and \ref{prop:weightstab-homology}.\end{proof}

\section{Homotopy automorphisms of connected sums}\label{sec:thm-B} 

Let $M$ be a closed and oriented $d$-dimensional manifold. For a non-empty finite set $S$, let $M_S = (\#^S M)\smallsetminus \mathring D^d$, be the space obtained by removing an open $d$-dimensional disk from the $S$-fold connected sum of $M$. If $S=\mathbf{n}=\{1,\ldots,n\}$ we simply write $M_n$. We then have a deformation retraction $M_S\xrightarrow \simeq \bigvee^S M_1$. Hence, there is an $\FI$-module given on objects by $S\mapsto \pi_k(\aut_*(M_S))$, as defined in the previous section. 

If we choose a base point in the boundary of $M_S$, there is an inclusion map $\aut_\partial(M_S)\to\aut_*(M_S)$ for every $S\in \FI$. In the first subsection of this section we prove that the $\FI$-module $S\mapsto\pi_k(\aut_*(M_S))$ lifts to an $\FI$-module given on objects by $S\mapsto\pi_k(\aut_\partial(M_S))$. In the second subsection we prove that $S\mapsto\pi_k^\Q(\aut_\partial(M_S))$ is a finitely generated $\FI$-module using certain rational models.

\subsection{The integral $\FI$-module structure on the homotopy automorphisms of iterated connected sums}\label{subsec:integral-thm-b} For the purposes of this section, we give an explicit construction of $M_n$ by removing the interiors of $n$ embedded little disks in $D^d$, which we fix as in Figure \ref{fig:stab1}, and gluing $n$ copies of $M\smallsetminus \mathring D^d$ along the new boundary components. Note that with this definition, we still have $M_1=M\smallsetminus\mathring D^d$. In Figure \ref{fig:stab2}, we see how we can embed $M_n$ into $M_{n+1}$ and by extending a boundary relative homotopy automorphism of $M_n$ by the identity thus defines a stabilization map
$$s_n:\aut_\partial(M_n)\to\aut_\partial(M_{n+1}).$$
In this section we will define a $\Sigma_n$-action on the homotopy groups of $\aut_\partial(M_n)$ and combining this with the stabilization induced by $s_n$, we obtain our $\FI$-module structure. Before we do this, we need to introduce some notation: 

\begin{figure}[h]
\centering
\begin{subfigure}[b]{0.5\textwidth}
  \centering
  \includegraphics[scale=0.62]{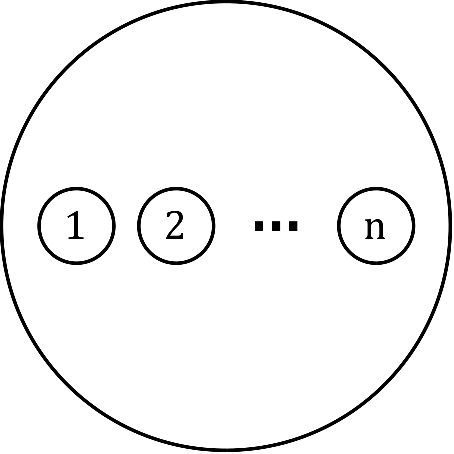}
  \caption{}
  \label{fig:stab1}
\end{subfigure}%
\begin{subfigure}[b]{0.5\textwidth}
  \centering
  \includegraphics[scale=0.7]{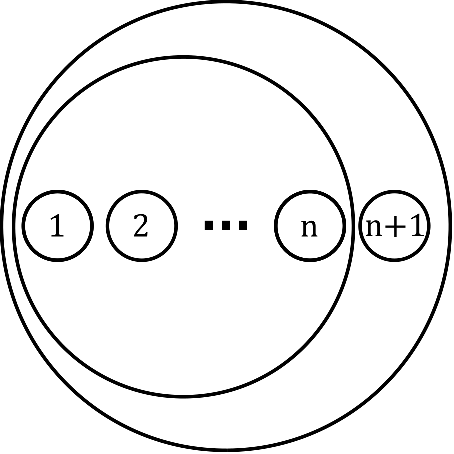}
  \caption{}
  \label{fig:stab2}
\end{subfigure}
\caption{We can define $M_n$ by gluing copies of $M_1$ into the disks in (A), while (B) illustrates how to define an embedding $M_n\hookrightarrow M_{n+1}$.}
\label{fig:stab}
\end{figure}

\begin{dfn}
For any pointed space $X$ and any finite set $S$, let us write $Q_{S,X}\colon\Sigma(S)\to \pi_0(\aut_*(\bigvee^S X))$ for the group homomorphism given by sending $\sigma\in \Sigma(S)$ to the homotopy class of the automorphism $\X_X(\sigma)\colon \bigvee^S X \to \bigvee^S X$ described in the beginning of Section \ref{sec:wedge}. \hfill$\diamond$\end{dfn}

\begin{rmk}
Since $\pi_0(\aut_*(\bigvee^S X))$ acts on $\pi_k(\aut_*(\bigvee^S X))$, we get an induced $\Sigma(S)$ action on $\pi_k(\aut_*(\bigvee^S X))$ by above. This action coincides with the $\Sigma(S)$-action coming from the $\FI$-module structure discussed in Section \ref{sec:wedge}.\hfill$\diamond$
\end{rmk}

\begin{dfn}
The deformation retraction $M_S \to \bigvee^S M_1$ induces an equivalence $\aut_*(M_S)\xrightarrow\sim \aut_*(\bigvee^S M_1)$.
Composing this map with the inclusion $\aut_\partial(M_S)\hookrightarrow\aut_*(M_S)$, yields a map
$$
u\colon\aut_\partial(M_S)\to \aut_*\left(\bigvee^S M_1\right).
$$
that induces a group homomorphism $\pi_0(u)\colon\pi_0(\aut_\partial(M_S))\to \pi_0(\aut_*(\bigvee^S M_1))$.\hfill$\diamond$
\end{dfn}

The first thing we will show to construct our $\FI$-module is the following:

\begin{prop}\label{prop:pi-noll-contains-sigma-n}
Assuming $d\ge 3$, there is a group homomorphism  $\varepsilon_n\colon\Sigma_n\to\pi_0(\aut_\partial(M_n))$ such that $Q_{n,M_1}$ factors as $Q_{n,M_1}=\pi_0(u)\circ \varepsilon_n$.
\end{prop}

\begin{rmk} Since the group $\pi_0(\aut_\partial(M_n))$ acts on the higher homotopy groups $\pi_k(\aut_\partial(M_n))$, this means that there is a $\Sigma_n$-action on the higher homotopy groups of $\aut_\partial(M_n)$, which is non-trivial whenever $\varepsilon_n$ is nontrivial. This action together with the stabilization maps will define our $\FI$-module structure. \end{rmk}

We will prove this in a number of steps, so let us first describe the idea: Writing $D:=D^d$, we consider the subgroup $G_n\subseteq \Diff_\partial(D)$ consisting of diffeomorphisms which fix the embedded little disks in $D$ from Figure \ref{fig:stab1}, up to permutation. There is then a group homomorphism $\pi:G_n\to\Sigma_n$, given by sending a diffeomorphism to the permutation it induces on the little disks. We also get a group homomorphism $G_n\to\Diff_\partial(M_n)$, given by constructing $M_n$ as above, and mapping $f\in G_n$ to the boundary relative diffeomorphism of $M_n$ which is given by $f$ outside the $n$ glued in copies of $M_1$, and on $\sqcup^n M_1$ is given by $\pi(f)$. We will construct a group homomorphism $\Sigma_n\hookrightarrow \pi_0(G_n)$, which postcomposed with the maps
$$\pi_0(G_n)\to\pi_0(\Diff_\partial(M_n))\to \pi_0(\aut_\partial(M_n))$$
is the map $\varepsilon_n$ described in Proposition \ref{prop:pi-noll-contains-sigma-n}. Let us now give the proof in more detail:

\begin{proof}
    Our choice of embedded disks in $D$ defines an element $e\in\Emb(\sqcup^n D, D)$. Let $\bar{e}$ denote its image in the quotient $\Emb(\sqcup^n D,D)/\Sigma_n$, where we take the quotient of the action permuting the embedded disks. Restricting to the image of $e$ defines a map $\Diff_\partial(D)\to \Emb(\sqcup^n D,D)$, which is a Serre fibration. The quotient map $\Emb(\sqcup^n D,D)\to\Emb(\sqcup^n D,D)/\Sigma_n$ is a covering map, so the composition \begin{equation}\label{eq:fibration-diff-emb}
        p:\Diff_\partial(D)\to\Emb(\sqcup^n D,D)/\Sigma_n
    \end{equation}
    is also a Serre fibration. The fiber over $\bar{e}$ consists of the diffeomorphisms which restricted to the image of $e$ is a permutation, i.e.\ $\mathrm{fib}_p(\bar{e})=G_n$. We thus get a connecting homomorphism
    $$\delta:\pi_1(\Emb(\sqcup^n D,D)/\Sigma_n)\to \pi_0(G_n)$$
    in the long exact sequence of homotopy groups. We will therefore first show that there is an injective group homomorphism $\Sigma_n\hookrightarrow \pi_1(\Emb(\sqcup^n D,D))$.

    Note that if we let $C_n(\mathring D)$ denote the ordered configuration space of $n$ points in $\mathring D$, there is a map $\hat{\rho}:\Emb(\sqcup^n D,D)\to C_n(\mathring D),$
    given by restricting to the center of each embedded disk. This map also has a section $\hat{s}$, given by sending a configuration to an embedding of $n$ little disks, centered at the respective points and with radii all equal to the minimum distance between the points and between the points and the boundary of $D$, divided by three. We also get an induced map
    \begin{equation}\label{eq:map-emb-conf}
        \rho:\Emb(\sqcup^n D,D)/\Sigma_n\to C_n(\mathring D)/\Sigma_n
    \end{equation}
    on orbits, which has a section $s$ defined in the corresponding way. Let us use the notation $U_n(\mathring D):=C_n(\mathring D)/\Sigma_n$ for brevity. Since we have assumed that $d\ge 3$, we have that $\pi_1(U_n(\mathring D))\cong\Sigma_n$ and thus we get a homomorphism $\pi_1(s):\Sigma_n\cong\pi_1(U_n(\mathring D))\to \pi_1(\Emb(\sqcup^n D,D)/\Sigma_n)$. Furthermore, note that since $s$ is a section $\pi_1(\rho)\circ\pi_1(s)$ is the identity on $\pi_1(U_n(\mathring D))$ and so  $\pi_1(s)$ is injective.

    By composing with connecting homomorphism in the long exact sequence associated to $p$, we thus get a homomorphism $\Sigma_n\to\pi_0(G_n)$. In order to understand this map better, we describe the connecting homomorphism $\delta$ in more detail. If $\gamma$ is a loop in $\Emb(\sqcup^n D,D)/\Sigma_n$, based at $\bar{e}$, representing an element of $\pi_1(\Emb(\sqcup^n D,D)/\Sigma_n)$, it lifts to a path $\tilde\gamma$ in $\Diff_\partial(D)$, starting at $\mathrm{id}_D$, since $p$ is a Serre fibration. The connecting homomorphism sends the class of $\gamma$ to the connected component of $G_n$ containing $\tilde\gamma(1)$. If we consider the restriction of $\delta$ to the image of inclusion $\pi_1(s)$ above, we see that a permutation $\sigma$ is sent to the isotopy class of some diffeomorphism in $G_n$ which restricted to the little disks is precisely $\sigma$. If we finally consider the composite map
$$\Sigma_n\rightarrow\pi_0(G_n)\to\pi_0(\Diff_\partial(M_n))\to\pi_0(\aut_\partial(M_n))\to\pi_0(\aut_*(M_n))\cong\pi_0(\aut_*(\vee^n M_1)),$$
it follows by the definition of the map $G_n\to\Diff_\partial(M_n)$ that this takes a permutation to the homotopy class of the homotopy automorphism of $\vee^n M_1$ given by permuting the wedge summands in the corresponding way. In other words, the composition is equal to $Q_{n,M_1}$, so we can simply define $\varepsilon_n$ as the composition of the first three maps. 
\end{proof}

\begin{rmk}\label{rmk:nontrivhomology}
If we assume that $M_1$ has non-trivial homology, then for any non-trivial permutation $\sigma$ we have that $\X_{M_1}(\sigma)\colon \bigvee^n M_1\to \bigvee^n M_1$ is not homotopic to the identity since it induces a non-trivial permutation of the reduced homology $\tilde H_*(\bigvee^n M_1) = \bigoplus_\mathbf{n} \tilde H_*(M_1)$, which is different from the identity map whenever $\tilde H_*(M_1)$ is non-trivial. If that is the case, the homomorphism $Q_{n,M_1}$ is injective, so it follows that $\varepsilon_n$ is injective as well and thus both $\pi_0(\aut_\partial(M_n))$ and $\pi_0(\aut_*(\bigvee^n M_1))$ contain a subgroup isomorphic to $\Sigma_n$.
\end{rmk}

\begin{cor} Under the assumptions of Remark \ref{rmk:nontrivhomology}, fix a subspace $A\subseteq \partial M_n$, possibly empty, such that $A\subset M_n$ is a cofibration. Then all of the groups $\pi_0(\aut_A(M_n))$, $\pi_0(\Diff_A(M_n))$ and  $\pi_0(\Homeo_A(M_n))$ contain a subgroup isomorphic to $\Sigma_n$. 
\end{cor}

\begin{proof}
Suppose that $A\neq\varnothing$ and let us first consider the case of homotopy automorphisms. Then the map $u:\aut_\partial(M_n)\to\aut_*\left(\bigvee^nM_1\right)$ factors as
$$\aut_\partial(M_n)\to\aut_A(M_n)\to\aut_*\left(\bigvee^nM_1\right),$$
proving this case. To get the cases with diffeomorphism or homeomorphisms, we consider the factorization
$$\Diff_\partial(M_n)\to\Diff_A(M_n)\to\Homeo_A(M_n)\to\aut_A(M_n)\to\aut_*(\vee^n M_1).$$

For the case where $A$ is empty, we instead postcompose with the map $\aut_*\left(\bigvee^n M_1\right)\to\aut\left(\bigvee^n M_1\right)$, and the resulting map factors as
$$\aut_\partial(M_n)\to\aut(M_n)\to\aut\left(\bigvee^nM_1\right).$$
The composition of $Q_{n,M_1}$ with the map induced on $\pi_0$ by the rightmost map above will still be injective and from this the case follows. To get the statement for diffeomorphisms and homeomorphisms, we instead use the factorization
$$\Diff_\partial(M_n)\to\Diff(M_n)\to\Homeo(M_n)\to\aut(M_n)\to\aut(\vee^n M_1).$$
\end{proof}

\begin{rmk}
    It was remarked by a referee that the 
 existence of the homomorphism $\Sigma_n\to\pi_0(\aut_\partial(M_n))$ is likely a consequence of a higher structure. More specifically, it is reasonable to expect that the space $\bigsqcup_{n\ge 1} B\aut_\partial(M_n)$ can be endowed with the structure of an $E_d$-algebra, i.e.\ an algebra over the little $d$-disks operad, in a similar way as for example the space $\bigsqcup_{n\ge 1} B\Diff_\partial(M_n)$. If this is the case, the $E_d$-algebra structure maps in particular give us a map
    $$E_d(n)/\Sigma_n\to B\aut_\partial(M_n),$$
    and since $E_d(n)/\Sigma_n\simeq U_n(\mathring D)$, taking fundamental groups gives us a map $\Sigma_n\to\pi_0(\aut_\partial(M_n))$, which should be precisely $\varepsilon_n$. We expect this to be true, but have elected to use a more hands on approach, since rigorously constructing the $E_d$-algebra structure is non-trivial and seems to require using methods from higher homotopy theory that go quite far beyond the scope of this paper. For comparison, what makes this easier in the case of diffeomorphisms is that we have a good model for $B\Diff_\partial(M_n)$ as a topological space, in terms of embeddings of $M_n$ into $\R^\infty$ (with certain boundary conditions), modulo the action of $\Diff_\partial(M_n)$. In contrast, it is not clear how to do a similar construction for homotopy automorphisms.
\end{rmk}

We have now defined the $\Sigma_n$-action on the homotopy groups of $\aut_\partial(M_n)$. Next we show that this action is compatible with the stabilization maps $s_n$.

\begin{prop}\label{prop:SigmaNinclusion}
There is a commutative diagram
\[\begin{tikzcd}
\Sigma_n\arrow[d,"\varepsilon_n"']\arrow[r]&\Sigma_{n+1}\arrow[d,"\varepsilon_{n+1}"]\\
\pi_0(\aut_\partial(M_n))\arrow[r,"\pi_0(s_n)"']&\pi_0(\aut_\partial(M_{n+1}))
\end{tikzcd}\]

where the upper horizontal map is the standard inclusion.
\end{prop}

\begin{proof}
We can construct stabilization maps $G_n\to G_{n+1}$, $\Emb(\sqcup^n D,D)/\Sigma_n\to\Emb(\sqcup^{n+1} D,D)/\Sigma_{n+1}$ and $C_n(\mathring D)/\Sigma_n\to C_{n+1}(\mathring D)/\Sigma_{n+1}$ in the same way as we defined $s_n:\aut_\partial(M_n)\to\aut_\partial(M_{n+1})$, using Figure \ref{fig:stab2}. This gives us a diagram
\[\begin{tikzcd}
\Sigma_{n}\arrow[r]\arrow[d,"\cong"]&\Sigma_{n+1}\arrow[d,"\cong"]\\
    \pi_1(U_n(\mathring D)\arrow[r]\arrow[d]&\pi_1(U_{n+1}(\mathring D))\arrow[d]\\
    \pi_1(\Emb(\sqcup^n D,D)/\Sigma_n)\arrow[r]\arrow[d]&\pi_1(\Emb(\sqcup^{n+1}D,D)/\Sigma_{n+1})\arrow[d]\\
    \pi_0(G_n)\arrow[r]\arrow[d]&\pi_0(G_{n+1})\arrow[d]\\
    \pi_0(\aut_\partial(M_n))\arrow[r]&\pi_0(\aut_\partial(M_{n+1})),
\end{tikzcd}\]
where the top horizontal arrow is the standard inclusion. The two upper squares, as well as the bottom square, are all commutative by the definition of the stabilization maps. The  second square from the bottom can be shown to be commutative simply by once again considering the definition of the connecting homomorphism in detail as above, but we can also reason as follows: We can define a map $\Diff_\partial(D)\to\Diff_\partial(D)$ in the same was as we defined the stabilization maps, using Figure \ref{fig:stab2} and extending by identity (note however that this map is homotopic to the identity), giving us a commutative diagram
\[\begin{tikzcd}
\Diff_\partial(D)\arrow[r]\arrow[d]&\Diff_\partial(D)\arrow[d]\\
    \Emb(\sqcup^{n}D,D)/\Sigma_{n}\arrow[r]& \Emb(\sqcup^{n+1}D,D)/\Sigma_{n+1},
\end{tikzcd}\]
which is a map of Serre fibrations. By functoriality, this induces a map between the long exact sequences of homotopy groups, in which the square we consider appears.
\end{proof}

\begin{cor}\label{cor:consistent}
For $k\geq 1$, the sequence $\{\pi_k(\aut_\partial(M_n)),\pi_k(s_n)\}$ is a consistent sequence of $\Z[\Sigma_n]$-modules. 
\end{cor}

\begin{proof}
Recall that $\pi_0(\aut_\partial(M_n))$ acts on $\pi_k(\aut_\partial(M_{n}))$ and thus, through the stabilization map 
$$\pi_0(s_n)\colon \pi_0(\aut_\partial(M_n))\to\pi_0(\aut_\partial(M_{n+1})),$$
$\pi_0(\aut_\partial(M_n))$ acts on $\pi_k(\aut_\partial(M_{n+1}))$ as well. By definition of the stabilization map, $\pi_k(s_n)$ is $\pi_0(\aut_\partial(M_n))$-equivariant.  

By considering $\pi_k(\aut_\partial(M_n))$ as a $\Z[\Sigma_n]$-module via the homomorphism $\varepsilon_n:\Sigma_n\to\pi_0(\aut_\partial(M_n))$, it follows from Proposition \ref{prop:SigmaNinclusion} and the equivariance discussed above that $\{\pi_k(\aut_\partial(M_n)),\pi_k(s_n)\}$ is a consistent sequence of $\Z[\Sigma_n]$-modules. 
\end{proof}

\begin{thm}\label{thm:thm-B-(a)}
For each $k\geq 1$, the $\FI$-module $S\mapsto \pi_k(\aut_*(M_S))\cong \pi_k\left(\aut_*\left({\bigvee}^S M_1\right)\right)$ lifts to an $\FI$-module
$$
S\mapsto \pi_k(\aut_\partial(M_S))
$$
where the standard inclusion $\mathbf{n}\hookrightarrow \mathbf{n+1}$ gives the induced stabilization map 
$$\pi_k(s)\colon\pi_k(\aut_\partial(M_n))\to \pi_k(\aut_\partial(M_{n+1})).$$
\end{thm}

\begin{proof}
We have shown in Corollary \ref{cor:consistent} that the homotopy groups $\{\pi_k(\aut_\partial(M_n))\}_{n\ge 1}$ form a consistent sequence of $\Z[\Sigma_n]$-modules and from the previous discussion, it is clear that the maps $\aut_\partial(M_n)\to\aut_*\left(\bigvee^nM_1\right)$ induce a map of consistent sequences to $\{\pi_k\left(\aut_*\left(\bigvee^nM_1\right)\right)\}_{n\ge 1}$, which we know comes from an $\FI$-module. Thus, it is sufficient to show that $\{\pi_k(\aut_\partial(M_n))\}_{n\ge 1}$ also comes from an $\FI$-module.

From Lemma \ref{lemma:consistentSeqFromFI-module}, we know that it suffices to show that if $\sigma\in\Sigma_{n+m}$ is such that $\sigma|_{\mathbf{n}}=\mathrm{id}$, it acts trivially on the image of the stabilization map $\pi_k(\aut_\partial(M_n))\to\pi_k(\aut_\partial(M_{n+m}))$.
Embedding $M_n$ in $M_{n+m}$ according to the composition of the embeddings $M_n\hookrightarrow\cdots\hookrightarrow M_{n+m}$ defined by Figure \ref{fig:stab2}, we may represent $\sigma$ by an automorphism $f_\sigma\in\aut_\partial(M_{n+m})$ which is supported completely on $M_m\subset M_{m+n}$ and is thus the identity on $M_n\subset M_{m+n}$. Any homotopy automorphism $g\in\mathrm{im}(s_{n+m-1}\cdots s_n\colon\aut_\partial(M_n)\to\aut_\partial(M_{m+n}))$ is supported on $M_n$, so we have that $f_\sigma gf^{-1}_\sigma = g$. From this we conclude that $\sigma$ on acts trivially on the image of the stabilization map $\pi_k(\aut_\partial(M_n))\to\pi_k(\aut_\partial(M_{n+m}))$.
\end{proof}

\subsection{Rational representation stability via algebraic models for relative homotopy automorphisms}\label{subsec:rat-stab-two}

We will study a certain dg Lie model for $B\aut_{\partial,\circ}(M_n)$  constructed in \cite{BM14} and use it in order to prove that the $\FI$-module $S\mapsto \pi_k^\Q(\aut_\partial(M_S))= \pi_k(\aut_\partial(M_S))\ot \Q$ is finitely generated. 

 We recall that a  quasi-free dg Lie algebra $(\L(V),d)$ is said to be minimal if $d(V)\subset [\L(V),\L(V)]$. If two minimal dg Lie algebras are quasi-isomorphic then they are isomorphic. Moreover, if $\L(V)$ is a minimal dg Lie algebra model for a nilpotent space $X$ of finite type, then one can show that $V$ is isomorphic to the desuspension of the reduced rational homology of $X$, which we will denote by $s^{-1}\tilde H_*(X;\Q)$.

In this subsection we fix a $d$-dimensional simply connected oriented closed manifold $M$ where $M_1 = M\smallsetminus\mathring D$ has a non-trivial rational homology. The intersection form on $H_*(M)$ induces a graded symmetric inner product of degree $d$ on the reduced homology $\tilde H_*(M_1)$. This in turn induces a graded anti-symmetric inner product of degree $d-2$ on $H=s^{-1}\tilde H^*(M_1)$.

\begin{dfn}
Let $H$ be a graded anti-symmetric inner product space of degree $d-2$ (e.g. $s^{-1}\tilde H_*(M_1)$) with a basis $\{\alpha_1,\dots, \alpha_m\}$. The dual basis $\{\alpha_1^\#,\dots,\alpha_m^\#\}$ is characterized by the following property
$$
\langle\alpha_i,\alpha_j^\#\rangle = \delta_{ij}.
$$
Let $\omega_H \in \L^2(H)$ be given by
$$
\omega_H = \frac12 \sum_{i=1}^m[\alpha_i^\#,\alpha_i].
$$
It turns out that $\omega_H$ is independent of choice of basis $\{\alpha_1,\dots,\alpha_m\}$; see \cite{BM14} for details. \hfill $\diamond$
\end{dfn}

\begin{rmk}
By same arguments as above the graded vector space $s^{-1}\tilde H_*(M_n)$ has also a structure of a graded anti-symmetric inner product space of degree $d-2$, which coincide with the one given by the direct sum $(s^{-1}\tilde H_*(M_1))^{\oplus n}$. \hfill $\diamond$
\end{rmk}
The next proposition is due to Stasheff \cite{stasheff83}, and discussed in \cite{BM14}.
\begin{prop}[\text{\cite[Theorem 2]{stasheff83}}, \text{\cite[Theorem 3.11]{BM14}}]
Let $M= M^d $ be a closed oriented $d$-dimensional manifold, let $M_1 = M\smallsetminus\mathring D$ and let $H = s^{-1}\tilde H_*(M_1)$. Then inclusion $S^{d-1}\cong\partial M_1\hookrightarrow M_1$ is modelled by a dg Lie algebra map
$$\iota\colon\L(x)\hookrightarrow \L(H)$$
$$\iota(x)= (-1)^d\omega_H,
$$
where $\L(H)$ and $\L(x)$ denote the minimal dgl models for $M_1$ and  $S^{d-1}$ respectively.\hfill $\diamond$
\end{prop}

Given a fixed basis $\{\alpha_1,\dots,\alpha_m\}$ for $H=s^{-1}\tilde H_*(M_1)$ we get a basis for $s^{-1}\tilde H_*(M_n)\cong H^{\oplus n}$, which is of the form
$$
\{\alpha_{i}^j\ |\ 1\leq i \leq m, 1\leq j\leq n\}.
$$
We denote $\omega_{H^{\oplus n}} = \frac12\sum_{i,j}[(\alpha_{i}^j)^\#,\alpha_{i}^j]\in \L(H^{\oplus n})$ by $\omega_n$. We have that $\omega_n$ is  invariant under the $\Sigma_n$-action on $\L(H^{\oplus n})$ that permutes the summands of $H^{\oplus n}$.

Note that $\iota\colon \L(x)\to \L(H^{\oplus n})$ is not a cofibration. In order to model the inclusion $\partial M_n\subset M_n$ by a cofibration in the model category of dg Lie algebras we need a new model for $M_n$. 

\begin{lemma}\label{lemma:cofibration-model-for-incl}
Let $\L(H^{\oplus n},x,y)$ be the dg Lie algebra that contains $\L(H^{\oplus n})$ as a dg  Lie subalgebra and where $|x|=d-2$ and $|y|=d-1$ and where 
$$ dx = 0,\quad dy = x-(-1)^d\omega_n.$$
Then 
$$\hat\iota \colon \L(x)\to\L(H^{\oplus n},x,y)$$
$$\hat\iota(x)=x$$
is a cofibration that models the  inclusion of $\partial M_n \cong S^{d-1}$ into $M_n$. Moreover this model is a  relative minimal model in the sense of \cite{ES20}.\hfill $\diamond$
\end{lemma}
\begin{proof}
The dg Lie algebra map $\rho\colon \L(H^{\oplus n},x,y) \to \L(H^{\oplus n})$ where $\rho|_{H^{\oplus n}} = \id_{H^{\oplus n}}$, $\rho(x)= (-1)^d\omega_n$ and $\rho(y)=0$ is a quasi-isomorphism. Straightforward computation shows that $\rho \circ \hat\iota = \iota$, proving that $\hat\iota$ is a model for $\iota$ (which is a model for the inclusion of the boundary). Minimality is straightforward verification; see \cite[§ 3]{ES20}.
\end{proof}

%\begin{notation} We will henceforth use both the models $\L(H^{\oplus n})$ and $\L(H^{\oplus n},x,y)$ for $M_n$, so we simplify the notation by setting $\L_n = \L(H^{\oplus n})$ and $\widetilde{\L}_n = \L(H^{\oplus n},x,y)$. \hfill $\diamond$\end{notation}

We have by Proposition \ref{prop:models-for-haut} (b), that a dg Lie algebra model for $B\aut_{\partial,\circ}(M_n)$ is given by\break $\Der^+(\L(H^{\oplus n},x,y)\|\L(x))$. However, we will use another model thanks to the following result by Berglund and Madsen:

\begin{prop}[\text{\cite[Theorem 3.12]{BM14}}]\label{prop:omega-relative-der}  Let $\Der(\L(H^{\oplus n})\|\omega_n)$ denote the dg Lie algebra of derivations on $\L(H^{\oplus n})$ that vanish on $\omega_n$ and where the differential is given by $[d_{\L(H^{\oplus n})},-]$. Then there is an  equivalence of dg Lie algebras 
\begin{align*}
    \Der^+(\L(H^{\oplus n})\|\omega_n)&\to \Der^+(\L(H^{\oplus n},x,y)\|\L(x)),\\
    \theta&\mapsto \hat{\theta}
\end{align*}
where 
$\hat \theta|_{\L(H^{\oplus n})}=\theta$ and $\theta(x)=\theta(y)=0$.    \hfill$\diamond$\end{prop}

\begin{rmk}
It follows that $\Der^+(\L(H^{\oplus n})\|\omega_n)$ is a dg Lie algebra model for $B\aut_{\partial,\circ}(M_n)$ and the inclusion 
$\Der^+(\L(H^{\oplus n})\|\omega_n)\to \Der^+(\L(H^{\oplus n}))$ is a model for the map $B\aut_{\partial,\circ}(M_n)\to B\aut_{*,\circ}(M_n)$, induced by the inclusion $\aut_{\partial,\circ}(M_n)\hookrightarrow \aut_{*,\circ}(M_n)$.\hfill$\diamond$
\end{rmk}

\begin{dfn}\label{dfn:FI-module-rel-der} With the terminology established in § \ref{sec:fi}, we define a  dg $\FI$-Lie algebra $\Der(\L\mathcal{H}\|\omega_{\mathcal H})$ as follows: for $S\in\FI$, we let $\Der(\L\mathcal{H}\|\omega_{\mathcal{H}})(S):=\Der(\L \mathcal{H}(S)\|\omega_S)$ be the dg Lie algebra of derivations on $\L \mathcal{H}(S) = \L(H^{\oplus S})$ that vanish on $\omega_S$. For $i\colon S\hookrightarrow T$ in $\FI$, we get a map $$\Der(\L\mathcal{H}\|\omega_\mathcal{H})(i)\colon \Der(\L (H^{\oplus S})\|\omega_S)\hookrightarrow\Der(\L (H^{\oplus{T}})\|\omega_T),$$ defined as follows: suppose $x_\alpha\in\mathcal{H}(T)$ lies in the direct summand of $\mathcal{H}(T)$ corresponding to $\alpha\in T$ and let $D\in \Der(\L (H^{\oplus S})\|\omega_S)$. Then $\Der(\L\mathcal{H}\|\omega_{\mathcal H})(i)D$ is determined by
$$
    \Big(\Der(\L\mathcal{H}\|\omega_\mathcal{H})(i)D\Big)(x_\alpha)=\begin{cases} 0 &\text{ if }\alpha\in T\setminus i(S),\\
    \left(\L\mathcal{H}(i)\circ D\circ \mathcal{H}(i)^{-1}\right)(x_\alpha) &\text{ if } \alpha\in i(S).\end{cases}
$$\hfill $\diamond$
\end{dfn}
We conclude from having such a dg $\FI$-Lie algebra the following:

\begin{rmk}\label{rmk:FI-module-Hk(Der|omega)}
The above dg $\FI$-Lie algebra structure induces an $\FI$-module structure on the homology. For $k\geq 1$, we have that $H_k(\Der(\L(H^{\oplus S})\|\omega_{S})) \cong \pi_k^\Q(\aut_\partial(M_S))$, which gives an $\FI$-module structure on $\{\pi_k^\Q(\aut_\partial(M_S))\}_{S\in\FI}$. We will show that this $\FI$-module structure coincides with the one obtained by rationalizing the $\FI$-module structure on $\{\pi_k(\aut_\partial(M_S))\}_{S\in\FI}$ defined in § \ref{subsec:integral-thm-b}.\hfill $\diamond$
\end{rmk}

\begin{prop}\label{prop:rat-stab}
A dg Lie algebra model for the stabilization map $B\aut_{\partial,\circ}(M_n)\to B\aut_{\partial,\circ}(M_{n+1})$ is given by 
     $$
     \varphi_n\colon\Der^+(\L(H^{\oplus n})\|\omega_n)\to \Der^+(\L(H^{\oplus n+1})\|\omega_{n+1})
     $$
     where $\varphi_n(\theta)$ coincides with $\theta$ on the  first $n$ summands of  $H^{\oplus n+1}$ and vanishes on the last summand.
\end{prop}
\begin{proof}
The proof is omitted since it is very similar to the proof of Proposition \ref{prop:models-for-stab-sigma}.
\end{proof}

\begin{prop}\label{prop:rat-Sigma_n-act}
The $\Sigma_n$-action on $\pi_*^\Q(\aut_\partial(M_n))$ induced by the group homomorphism $\varepsilon_n\colon \Sigma_n\to \pi_0(\aut_\partial(M_n))$ is modelled by the $\Sigma_n$-action on $H_k(\Der(\L(H^{\oplus n}))\|\omega_n)$ induced by the $\FI$-module structure defined in Definition \ref{dfn:FI-module-rel-der}. 
\end{prop}

\begin{proof}
For every $\sigma\in\Sigma_n$, let $\eta_\sigma\in \aut_\partial(M_n))$ denote a representative for $\varepsilon_n(\sigma)\in \pi_0(\aut_\partial(M_n))$. For every $\sigma\in\Sigma_n$ we define a self-equivalence 
$$
\mathrm{Ad}_\sigma\colon \aut_\partial(M_n)\to \aut_\partial(M_n)
$$
$$
\mathrm{Ad}_\sigma(f) = \eta_\sigma\circ f \circ \eta_{\sigma^{-1}}.
$$
This induces a $\Sigma_n$-action on $\pi_k(\aut_\partial(M_n))$ given by  $\sigma.a = \pi_k(\mathrm{Ad}_\sigma)(a)$  which is precisely the $\Sigma_n$-action given by the $\FI$-module structure.

As we saw in Lemma \ref{lemma:cofibration-model-for-incl}, $\hat\iota\colon \L(x)\to \L(H^{\oplus n},x,y)$ is a minimal relative model for the inclusion $\partial M_n\hookrightarrow M_n$. 

By Lemma \ref{lemma:espicsaleh}  $\eta_\sigma$ is modelled by an $\hat\iota$-relative automorphism  $\zeta_\sigma\in\Aut_{\hat\iota}(\L(H^{\oplus n},x,y))$, and hence, by Corollary \ref{cor:F-is-equivariant}, the automorphism
$$
\alpha_{\zeta_\sigma}\colon \Der(\L(H^{\oplus n},x,y)\|\L(x))\to \Der(\L(H^{\oplus n},x,y)\|\L(x))
$$
$$
\alpha_{\zeta_\sigma}(\theta) = \zeta_\sigma\circ\theta\circ\zeta^{-1}_\sigma
$$
is a model for the delooping of $\mathrm{Ad}_{\sigma}$.
In particular $H_k(\alpha_{\zeta_\sigma})$ is a model for $\pi_k(\mathrm{Ad}_\sigma)$. Moreover, this defines a $\Sigma_n$-action on $H_k(\Der(\L(H^{\oplus n},x,y)\|\L(x)))$ given by $\sigma.b = H_k(\alpha_{\zeta_\sigma})(b)$ that models the $\Sigma_n$-action on $\pi_k^\Q(\aut_\partial(M_n))$ described above.

Since the isomorphism of Lemma \ref{lemma:espicsaleh} is not explicit, we do not know what $\zeta_\sigma$ is.  However, viewing $\zeta_\sigma$ as a non-relative automorphism that models pointed homotopy automorphisms, we know that it models the permutation of the summands of $\bigvee_{i=1}^nM_1$ corresponding to $\sigma\in \Sigma_n$. A model for this pointed map is given by $\psi_\sigma\colon \L(H^{\oplus n},x,y)\to\L(H^{\oplus n},x,y)$ where $\psi_\sigma(\alpha_i^j)= \alpha_i^{\sigma(j)}$, $\psi_\sigma(x)=x$  and  $\psi_\sigma(y) =y$. Since $\psi_\sigma$ and $\zeta_\sigma$ model the same pointed homotopy class of pointed maps they have to be  homotopic as dg Lie algebra maps and thus $\alpha_{\zeta_\sigma}$ and $\alpha_{\psi_\sigma}$ induce the same map on the homology of $\Der(\L(H^{\oplus n},x,y))$. In particular, for every cycle $\theta\in Z(\Der(\L(H^{\oplus n},x,y)))$, the difference $\alpha_{\zeta_\sigma}(\theta)-\alpha_{\psi_\sigma}(\theta)$ is a boundary $\partial \nu$ for some $\nu\in \Der(\L(H^{\oplus n},x,y))$.

Note that $\psi_\sigma$ is also $\hat\iota$-relative, but not necessarily $\hat\iota$-equivalent to $\zeta_\sigma$. Since $\zeta_\sigma$ and $\psi_\sigma$ are $\hat\iota$-relative, $\alpha_{\zeta_\sigma}$ and  $\alpha_{\psi_\sigma}$ define automorphisms of $\Der(\L(H^{\oplus n},x,y)\|\L(x))$. We will show that these automorphisms induce the same map on homology. Given a cycle $\theta\in Z(\Der(\L(H^{\oplus n},x,y)\|\L(x)))$, we have that $\theta$ is also a cycle in $\Der(\L(H^{\oplus n},x,y))$, and thus, by above there is some $\nu\in \Der(\L(H^{\oplus n},x,y))$ such that $\alpha_{\zeta_\sigma}(\theta)-\alpha_{\psi_\sigma}(\theta) = \partial \nu$. By this equality it follows that  that $\partial \nu(x)=0$.

Let $\widetilde\nu\in \Der(\L(H^{\oplus n},x,y)\|\L(x))$ be given by $\widetilde \nu|_{\mathrm{span}(H^{\oplus n},y)} = \nu|_{\mathrm{span}(H^{\oplus n},y)}$ and $\widetilde \nu(x)=0$. Now it is straightforward to see that 
$$
\alpha_{\zeta_\sigma}(\theta)-\alpha_{\psi_\sigma}(\theta) = \partial \nu = \partial \widetilde \nu.
$$
Hence $\alpha_{\zeta_\sigma}$ and $\alpha_{\psi_\sigma}$ induce the same morphisms on $H_*(\Der(\L(H^{\oplus n},x,y)|\L(x)))$. From this we conclude that the $\Sigma_n$-action on $H_k(\Der(\L(H^{\oplus n},x,y)|\L(x)))$ given by $\sigma.b = H_k(\alpha_{\psi_\sigma})(b)$ is a model for the $\Sigma_n$-action on $\pi_k^\Q(\aut_\partial(M_n))$.

Now consider the $\omega_n$-preserving automorphism of $\phi_\sigma\colon \L(H^{\oplus n})\to \L(H^{\oplus n})$ given by $\phi_\sigma = \psi_\sigma|_{\L(H^{\oplus n})}$. 
This yields an automorphism
$$
\alpha_{\phi_\sigma}\colon \Der(\L(H^{\oplus n})\|\omega_n)\to \Der(\L(H^{\oplus n})\|\omega_n)
$$
$$
\alpha_{\phi_\sigma}(\theta) = \phi_\sigma\circ\theta\circ\phi_\sigma^{-1}.
$$
The $\Sigma_n$-action on $\Der(\L(H^{\oplus n})\|\omega_n)$ given by $\sigma.b = \alpha_{\phi_\sigma}(b)$ is  the same $\Sigma_n$-action coming from the $\FI$-module structure described in  Definition \ref{dfn:FI-module-rel-der}.

The diagram
$$
\xymatrix{
\Der^+(\L(H^{\oplus n})\|\omega_n)\ar[d]_\sim\ar[r]^{\alpha_{\phi_\sigma}}&\Der^+(\L(H^{\oplus n})\|\omega_n)\ar[d]^\sim
\\
\Der^+(\L(H^{\oplus n},x,y)\|\L(x))\ar[r]_{\alpha_{\psi_\sigma}}&\Der^+(\L(H^{\oplus n},x,y)\|\L(x)),
}
$$
where the vertical maps are the quasi-isomorphisms of dg Lie algebras described in Propostion \ref{prop:omega-relative-der}, is commutative, which gives that the induced $\Sigma_n$-action on $H_k(\Der^+(\L(H^{\oplus n})\|\omega_n))$ is a model for the $\Sigma_n$-action on $H_k(\Der(\L(H^{\oplus n},x,y)|\L(x)))$ (which in turn, is a model for the $\Sigma_n$-action on $\pi_k^\Q(\aut_\partial(M_n))$).
\end{proof}

We recall that the Lie operad $\mathscr Lie$ is a cyclic operad, i.e.\ that the $\Sigma_n$-action on $\mathscr Lie(n)$ extends to a $\Sigma_{n+1}$-action. Let $\mathscr Lie_c(n+1)$ denote $\mathscr Lie(n)$ viewed as a $\Sigma_{n+1}$-representation. 

\begin{prop}[\text{\cite[Proposition 6.6]{BM14}}]\label{prop:der-cyclic-lie}
There is an isomorphism of $\FI$-modules
$$
\Der(\L\mathcal{H}\|\omega_{\mathcal H}) \cong s^{2-d}\mathbb S_{\mathscr Lie_c}(\mathcal H).
$$
\end{prop}
\begin{proof}
We will prove that this isomorphism is a special case of a more general isomorphism proved in \cite{BM14}. In loc.\ cit., the authors are considering the category of  graded anti-symmetric inner product
spaces of degree $2-d$ and with morphisms being linear maps of degree 0 that preserves the inner product. They call this category for $\mathrm{Sp}^{2-d}$. An $\mathrm{Sp}^{2-d}$-module is a functor from $\mathrm{Sp}^{2-d}$ to the category of graded vector spaces. By \cite[Proposition 6.6]{BM14}, $V\mapsto \Der(\L(V)\| \omega_V)$ defines a $\mathrm{Sp}^{2-d}$-module that is isomorphic to the  $\mathrm{Sp}^{2-d}$-module given by
$V\mapsto s^{2-d}\mathbb S_{\mathscr Lie_c}(V)$.

Any morphism $i\colon S\to T$ in $\FI$, we have that the associated map $\mathcal H(i)\colon \mathcal H(S) = H^{\oplus S}\to \mathcal H(T) = H^{\oplus T}$ is a morphism of $\mathrm{Sp}^{2-d}$-modules. Thus the isomorphism above follows.\end{proof}

\theoremtwo*

\begin{proof}
(a): This is Theorem \ref{thm:thm-B-(a)}.

(b): By the isomorphism in Proposition \ref{prop:der-cyclic-lie} we have that 
$$\Der(\L\mathcal H \| \omega_{\mathcal{H}})_k \cong \mathbb S_{\mathscr Lie_c}(\mathcal H)_{k+d-2}.$$
By Proposition \ref{propSchurFT} we have that 
$$
\weight(\Der(\L\mathcal H \|\omega_{\mathcal{H}})_k)\leq k+d-2 
$$
and
$$
\stabdeg(\Der(\L\mathcal H \| \omega_{\mathcal{H}})_k)\leq k+d-2
$$
The weight and the stability degree for the homology follows from Proposition \ref{prop:weightstab-homology}.

\end{proof}

\printbibliography

\noindent
\Addresses

\end{document}